\documentclass[10pt]{article}

\usepackage{amsfonts,graphicx,amsmath,amssymb,amsthm,color,nicefrac,enumerate, layout, bbm, vmargin, listings,  hyperref}
\usepackage[latin1]{inputenc}

\newcommand{\R}{\mathbb{R}}

\newcommand{\N}{\mathbb{N}}
\newcommand{\Z}{\mathbb{Z}}
\newcommand{\E}{\mathbb{E}}

\renewcommand{\P}{\mathbb{P}}

\newcommand{\Xpsione}{X^{\psi,(1)}}
\newcommand{\Xpsitwo}{X^{\psi,(2)}}
\newtheorem{theorem}{Theorem}[section]
\newtheorem{lemma}[theorem]{Lemma}

\newtheorem{cor}[theorem]{Corollary}

\let\inf\relax \DeclareMathOperator*\inf{\vphantom{p}inf}

\let\lim\relax \DeclareMathOperator*\lim{\vphantom{p}lim}

\begin{document}

\title{ On stochastic differential equations with arbitrarily slow convergence rates for strong approximation in two space dimensions}

\author{M\'at\'e Gerencs\'er, Arnulf Jentzen, and Diyora Salimova}

\maketitle

\begin{abstract}
In the recent article [Jentzen, A., Müller-Gronbach, T., and Yaroslavtseva, L.,  Commun. Math. Sci., 14(6), 1477--1500, 2016] it has been established that for every arbitrarily slow convergence speed and every natural number $d \in \{4,5,\ldots\}$ there exist $d$-dimensional stochastic differential equations (SDEs) with infinitely often
differentiable and globally bounded coefficients such that no approximation method based on finitely many observations of the driving Brownian motion can converge in absolute mean to the solution faster than the given speed of convergence. In this paper we strengthen the above  result  by proving that this slow convergence phenomena also arises in two ($d=2$) and three ($d=3$) space dimensions.
\end{abstract}


\section{Introduction}

In the recent article \cite{jentzen2016slow} it has been established that for every arbitrarily slow convergence speed and every natural number $d \in \{4,5,\ldots\}$ there exist $d$-dimensional stochastic differential equations (SDEs) with infinitely often
differentiable and globally bounded coefficients such that no approximation method based on finitely many observations of the driving Brownian motion can converge in absolute mean to the solution faster than the given speed of convergence. More specifically, Theorem~1.3 in \cite{jentzen2016slow} implies the following theorem.

\begin{theorem}
	\label{thm:d4}
Let $T \in (0, \infty)$, $d \in \{4,5, \ldots\}$, $\xi \in \R^d$, $m \in \N = \{1, 2, \ldots\}$, $(\varepsilon_n)_{n \in \N} \subseteq (0, T]$, $(\delta_n)_{n \in \N} \subseteq [0,\infty)$ satisfy $\limsup_{n \to \infty} \delta_n = 0$. Then there exist infinitely often differentiable and globally bounded functions $\mu \colon \R^d \to \R^d$ and $\sigma \colon \R^{d} \to \R^{d \times m}$ such that for every probability space $(\Omega, \mathcal{F}, \P)$, every   normal filtration $\mathbb{F}=(\mathbb{F}_t)_{t \in [0,T]}$ on $(\Omega, \mathcal{F}, \P)$, every standard $(\Omega, \mathcal{F}, \P, \mathbb{F})$-Brownian motion $W \colon  [0,T] \times \Omega \to \R^m$, every continuous $\mathbb{F}$-adapted stochastic process $X \colon [0,T] \times \Omega \to \R^d$ with $\forall \, t \in [0,T] \colon \P(X_t = \xi + \int_0^t \mu(X_s) \, ds + \int_0^t \sigma(X_s) \, dW_s)=1$, and every $n \in \N$ it holds that 
\begin{align}
\label{eq:thm:d4}
 \inf_{t_1, \ldots, t_n \in [0,T]} \inf_{\substack{u \colon (\R^m)^n \times C([\varepsilon_n,T], \R^m)   \to \R^d \\  \text{ measurable}}}\E  \Big[\big\| X_T  - u( W_{t_1},  \ldots, W_{t_n},(W_s)_{s \in [\varepsilon_n,T]})\big\|_{\R^d}\Big] \geq \delta_n.
\end{align}
\end{theorem}
In this paper we strengthen the above result  by proving that for every arbitrarily slow convergence speed and every natural number $d \in \{2,3,\ldots\}$ there exist $d$-dimensional SDEs with infinitely often
differentiable and globally bounded coefficients such that no approximation method based on finitely many observations of the driving Brownian motion can converge in absolute mean to the solution faster than the given speed of convergence. More precisely, in this work we establish the following theorem.

\begin{theorem}
\label{thm:intro}
Let $T \in (0, \infty)$, $\tau \in (0,T)$, $d \in \{2,3, \ldots\}$, $\xi \in \R^d$, $m \in \N$, $(\varepsilon_n)_{n \in \N} \subseteq (0, \tau]$, $(\delta_n)_{n \in \N} \subseteq [0,\infty)$ satisfy $\limsup_{n \to \infty} \delta_n = 0$. Then there exist infinitely often differentiable and globally bounded functions $\mu \colon \R^d \to \R^d$ and $\sigma \colon \R^{d} \to \R^{d \times m}$ such that for every probability space $(\Omega, \mathcal{F}, \P)$, every   normal filtration $\mathbb{F}=(\mathbb{F}_t)_{t \in [0,T]}$ on $(\Omega, \mathcal{F}, \P)$, every standard $(\Omega, \mathcal{F}, \P, \mathbb{F})$-Brownian motion $W \colon  [0,T] \times \Omega \to \R^m$, every continuous $\mathbb{F}$-adapted stochastic process $X \colon [0,T] \times \Omega \to \R^d$ with $\forall \, t \in [0,T] \colon \P(X_t = \xi + \int_0^t \mu(X_s) \, ds + \int_0^t \sigma(X_s) \, dW_s)=1$, and every $n \in \N$ it holds that 
\begin{align}
\label{eq:intro}
\inf_{\substack{a,b \in [0, \tau], \\
b-a \geq \varepsilon_n }}  \inf_{t_1, \ldots, t_n \in [0,T]} \inf_{\substack{u \colon (\R^m)^n \times C([0,a] \cup [b,T], \R^m)   \to \R^d \\  \text{ measurable}}}\E  \Big[\big\| X_T - u( W_{t_1},  \ldots, W_{t_n}, (W_s)_{s \in [0,a] \cup [b,T]})\big\|_{\R^d}\Big] \geq \delta_n.
\end{align}
\end{theorem}
Theorem~\ref{thm:intro} follows immediately from Corollary~\ref{cor:W} below. In the following we provide a  brief and rough intuition behind the proof of Theorem~\ref{thm:intro} and we also comment on the new ideas used in the proof of Theorem~\ref{thm:intro} which allow  to reduce the dimensionality from $d=4$ in \cite[Theorem~1.3]{jentzen2016slow} (and Theorem~\ref{thm:d4} above, respectively) to $d=2$ in Theorem~\ref{thm:intro} in this work. A key aspect in both proofs (proof of \cite[Theorem~1.3]{jentzen2016slow} and proof of Theorem~\ref{thm:intro} in this work) is to construct the SDE for \cite[Theorem~1.3]{jentzen2016slow} and Theorem~\ref{thm:intro} in such a way that it admits different phases along the time evolution in which it behaves conceptually differently. In the first phase the SDE is designed in such a way  that all numerical schemes of the form appearing in \eqref{eq:thm:d4} and \eqref{eq:intro}, respectively, approximate the solution of the SDE strongly with a possibly small but non-neglectable error. The phases in the SDE thereafter are then employed to switch smoothly from the first phase to the last phase. The last phase, in turn, consists of the dynamics of an SDE which acts, roughly speaking, as a magnifying glass which increases the possibly small error in the first phase to an error with arbitrarily slow strong convergence speed. In the previous work \cite{jentzen2016slow} one of the components of the SDE has been employed to describe the time variable which, in turn, allows to timely switch between the different phases. A key idea of this work is to design the SDE in such a way  that the time variable is incorporated into the magnifying glass and thereby allowed us  to reduce the dimensionality of the SDE system.

Next we would like to point out that Theorem~\ref{thm:d4} and Theorem~\ref{thm:intro} both assume that the sequence $(\varepsilon_n)_{n \in \N}$ of real numbers appearing in Theorems~\ref{thm:d4} and \ref{thm:intro}, respectively,  is strictly positive. Note that this hypothesis can not be omitted as the solution $X_T$ is $\P$-almost surely equal to $u((W_s)_{s \in [0,T]})$ for some measurable function $u \colon C([0,T], \R^m) \to \R{^d}$ (cf., e.g., \eqref{eq:meas} in Lemma~\ref{lem:main} below). We also would  like to add that Theorem~\ref{thm:intro}, in particular, ensures that for every natural number $d \in \{2,3, \ldots\}$ there exist $d$-dimensional SDEs with infinitely often
differentiable and globally bounded coefficients such that no approximation method based on finitely many observations of the driving Brownian motion converges with any polynomial order of convergence. The precise statement of this fact is the subject of the following corollary of Theorem~\ref{thm:intro}.
\begin{cor}
\label{cor:intro}
Let $T \in (0, \infty)$, $\tau \in (0,T)$, $d \in \{2,3, \ldots\}$, $\xi \in \R^d$,  $(\varepsilon_n)_{n \in \N} \subseteq (0, \tau]$. Then there exist infinitely often differentiable and globally bounded functions $\mu \colon \R^d \to \R^d$ and $\sigma \colon \R^{d} \to \R^{d}$ such that for every probability space $(\Omega, \mathcal{F}, \P)$, every   normal filtration $\mathbb{F}=(\mathbb{F}_t)_{t \in [0,T]}$ on $(\Omega, \mathcal{F}, \P)$, every standard $(\Omega, \mathcal{F}, \P, \mathbb{F})$-Brownian motion $W \colon  [0,T] \times \Omega \to \R$, every continuous $\mathbb{F}$-adapted stochastic process $X \colon [0,T] \times \Omega \to \R^d$ with $\forall \, t \in [0,T] \colon \P(X_t = \xi + \int_0^t \mu(X_s) \, ds + \int_0^t \sigma(X_s) \, dW_s)=1$, and every $r \in (0,\infty)$ it holds that 
\begin{align}
\liminf_{n \to \infty} \left( n^r\inf_{\substack{a,b \in [0, \tau], \\
b-a \geq \varepsilon_n }}  \inf_{t_1, \ldots, t_n \in [0,T]} \inf_{\substack{u \colon \R^n \times C([0,a] \cup [b,T], \R)   \to \R^d \\  \text{ measurable}}}\E  \Big[\big\| X_T - u( W_{t_1},  \ldots, W_{t_n}, (W_s)_{s \in [0,a] \cup [b,T]})\big\|_{\R^d}\Big] \right) = \infty.
\end{align}
\end{cor}
Corollary~\ref{cor:intro} is a direct consequence of Theorem~\ref{thm:intro} above (choose $m=1$ and $\delta_n = \frac{1}{\ln(n+1)}$ for  $n \in \N$ in the notation of Theorem~\ref{thm:intro}). We also would like to point out that the main contribution of this work is to establish Theorem~\ref{thm:intro} in the case $d=2$ (cf.\ Corollary~\ref{cor:gen:C} below). Roughly speaking, the general case $d \in \{2,3,\ldots\}$ then follows from the case $d=2$ by filling up drift and diffusion coefficients  with zero entries. In addition, observe that in the deterministic case $(\sigma = ( \R^d \ni x \mapsto 0 \in \R^d))$ a slow convergence phenomena of the type \eqref{eq:intro} fails to hold as the standard Euler scheme is known to converge with order $1$ if $\mu$ is locally Lipschitz continuous and if a solution of the ordinary differential equation (ODE) does exist on the time interval $[0,T]$. 

Further lower error bounds for strong and weak numerical approximation schemes for SDEs with non-globally Lipschitz continuous coefficients can be found in \cite{hjk11,HutzenthalerJentzenKloeden2013,Hairer2015,jentzen2016slow,MullerGronbachYaroslavtseva2017,Yaroslavtseva2016}.  Hairer et al.~\cite[Theorem~1.3]{Hairer2015} and M\"uller-Gronbach \& Yaroslavtseva~\cite[Theorems 1--3]{MullerGronbachYaroslavtseva2017} deal with lower bounds for weak approximation errors and Yaroslavtseva~\cite[Corollary~2]{Yaroslavtseva2016} extends \cite[Theorem~1.3]{jentzen2016slow} (cf. also Theorem~\ref{thm:d4} above) to  numerical approximation schemes where the driving Brownian motion can be evaluated adaptively. Each of the references \cite{Hairer2015,jentzen2016slow,MullerGronbachYaroslavtseva2017,Yaroslavtseva2016} assumes beside other hypotheses that the dimension $d$ of the considered SDE satisfies $d \geq 4$. The main contribution of this work is to reveal that a slow convergence phenomena of the form \eqref{eq:intro}  also arises in two $(d=2)$ and three $( d=3)$ space dimensions. Upper error bounds and numerical approximation schemes for SDEs with non-globally Lipschitz continuous coefficients can, e.g., be found in \cite{Hu1996,Gyoengy1998,Higham2002,HutzenthalerJentzenKloeden2012,WangGan2013,Hutzenthaler2015,Sabanis2013,sabanis2016,TretyakovZhang2013} and  the references mentioned therein. Lower error bounds for strong approximation schemes for SDEs with globally Lipschitz continuous coefficients can, e.g., be found in the overview article M\"uller-Gronbach \& Ritter~\cite{Mueller-Gronbach2008}  and the references mentioned therein.

A fundamental long term goal in the numerical analysis of SDEs is to characterize strong/weak convergence rates for numerical approximations of SDEs in terms of explicit conditions on the coefficient functions of the SDE under consideration. In particular, it is of fundamental importance in this research area to reveal explicit conditions on the coefficients of the SDE which are both necessary and sufficient for numerical approximations to converge with positive strong/weak convergence rates. There are a number of articles in the literature which provide sufficient conditions for strong convergence rates for numerical approximations  (cf., e.g., \cite{Hu1996,Gyoengy1998,Higham2002,HutzenthalerJentzenKloeden2012,WangGan2013,Hutzenthaler2015,Sabanis2013,sabanis2016,TretyakovZhang2013} and  the references mentioned therein). These  conditions are far from being necessary for strong convergence rates. A key contribution of the lower bounds obtained in the above mentioned references \cite{ Hairer2015,jentzen2016slow,MullerGronbachYaroslavtseva2017,Yaroslavtseva2016} as well as in this work is to develop a better understanding of possible necessary and sufficient conditions for strong or weak convergence rates.


\section{Construction of the coefficients of the considered two-dimensional SDEs}
\label{sec:existence}

In this section we establish two elementary auxiliary results (see Lemma~\ref{lemma:f} and Lemma~\ref{lemma:g} below) which demonstrate that the functions $f, g \colon \R \to \R$ in \eqref{eq:setting:f} and \eqref{eq:setting:g} below have suitable regularity properties.

\subsection{Setting}
\label{setting:fg}

Let $T, \mu \in (0,\infty)$, $\tau, \tau_1 \in (0,T)$, $\tau_2 \in (\tau_1,T)$, $\varepsilon \in (0, \min\{T-\tau, \tau(1- 2^{\nicefrac{-1}{3}})\})$, $F ,\rho,h,  f, g\in C(\R,\R)$  satisfy for all $x \in \R$ that
\begin{align}
\tau_1 = \tau + \varepsilon, \qquad \mu = \int_{-\varepsilon}^{\varepsilon} \exp\!\Big(\tfrac{-1}{(\varepsilon^2 - t^2)}\Big) \, dt,
\end{align}
\begin{align}
F(x)  = \begin{cases}
4 \tau &\colon \quad x \leq -\tau \\
2\tau -2x &\colon \quad -\tau < x < \tau\\
0 &\colon \quad x \geq \tau
\end{cases},
\end{align}
\begin{align}
\rho(x)= \begin{cases} 
\tfrac{1}{\mu}  \exp\!\Big(\tfrac{-1}{(\varepsilon^2 - x^2)}\Big) & \colon  \quad |x| < \varepsilon \\
0 & \colon \quad  |x| \geq \varepsilon
\end{cases},
\end{align}

\begin{align}
h(x) = \begin{cases} 
\exp\!\left(-\tfrac{1}{x}\right) & \colon \quad x >0 \\
0 & \colon \quad  \, x \leq 0 
\end{cases},
\end{align}
\begin{align}
\label{eq:setting:f}
f(x)  = \int_{-\infty}^{\infty} \rho(t) \, F(x-t) \, dt,
\end{align}
and 
\begin{align}
\label{eq:setting:g}
\begin{split}
g(x) = \frac{4h(x-\tau_1)}{h(x-\tau_1)+ h(\tau_2-x)}.
\end{split}
\end{align}

\subsection{Properties of the function appearing in the first component of the considered two-dimensional SDE}

The next result, Lemma~\ref{lemma:f}, establishes a few elementary (regularity) properties of the function $f\colon \R \to \R$ in \eqref{eq:setting:f} in Section~\ref{setting:fg}.

\begin{lemma}
\label{lemma:f}
Assume the setting in Section~\ref{setting:fg}. Then
\begin{enumerate}[(i)]
\item \label{item:f:b} it holds that $\sup_{x \in \R} |f(x)| < \infty$,
\item \label{item:f:lb} it holds that $f((-\infty,\tau_1)) \subseteq (0,\infty)$,
\item \label{item:f:ub} it holds that $f([\tau_1,\infty)) = \{0\}$,
\item \label{item:f:0} it holds that $f'(\R) \subseteq [-2,0]$,
\item \label{item:f:-1} it holds that $f'((0,\tau)) \subseteq [-2, -1)$, and 
\item \label{item:f:int} it holds that $  \int_0^{\tau_1} |f(s)|^2 \, ds \geq \frac{2\tau^3}{3}$.
\end{enumerate}
\end{lemma}
\begin{proof}[Proof of Lemma~\ref{lemma:f}]
Throughout this proof let $\lambda \colon \mathcal{B}(\R) \to [0,\infty]$ be the Lebesgue-Borel measure on $\R$. Note that 
\begin{align}
\begin{split}
\sup_{x \in \R} |f(x)| &\leq \left[ \sup_{x \in \R} \big|F(x)\big| \right] \! \left[ \int_{-\infty}^{\infty} \rho(t) \, dt \right] =  4 \tau \left[ \int_{-\infty}^{\infty} \rho(t) \, dt \right]\\
&= \tfrac{4 \tau}{\mu} \int_{-\varepsilon}^{\varepsilon} \exp\!\Big(\tfrac{-1}{(\varepsilon^2 - t^2)}\Big) \, dt =  4\tau < \infty.
\end{split}
\end{align}
This establishes Item~\eqref{item:f:b}. Next note that for all $x \in \R$ it holds that
\begin{align}
\label{eq:f:int}
\begin{split}
f(x) & = \int_{-\infty}^{\infty} \rho(t) \, F(x-t) \, dt = \int_{-\varepsilon}^{\varepsilon} \rho(t) \, F(x-t)  \, dt\\
& = \int_{-\varepsilon}^{\varepsilon} \rho(t) \, F(x-t) \, \mathbbm{1}_{(-\infty,\tau)}(x-t) \, dt\\
& = \int_{-\varepsilon}^{\varepsilon} \rho(t) \, F(x-t) \,  \mathbbm{1}_{(x-\tau,\infty)}(t) \, dt\\
& = \int_{-\varepsilon}^{\varepsilon} \rho(t) \, F(x-t) \, \mathbbm{1}_{(x-\tau_1 + \varepsilon,\infty)}(t) \, dt.
\end{split}
\end{align}
This  proves Item~\eqref{item:f:ub}. Moreover, observe that for all $x \in (-\infty, \tau_1)$ it holds that
\begin{align}
\label{eq:lambda}
\lambda((-\varepsilon,\varepsilon) \cap (x-\tau_1 +\varepsilon, \infty)) > 0
\end{align}
and 
\begin{align}
\label{eq:rho:F}
\forall \, t \in (-\varepsilon,\varepsilon) \cap (x-\tau_1 +\varepsilon, \infty) \colon \rho(t) \, F(x-t) >0.
\end{align}
Combining \eqref{eq:lambda} and \eqref{eq:rho:F} with \eqref{eq:f:int} yields that  for all $x \in (-\infty, \tau_1)$ it holds that $f(x) > 0$. This establishes Item~\eqref{item:f:lb}. Next observe that  for all $x \in \R$ it holds that 
\begin{align}
\label{eq:f:der}
\begin{split}
f'(x) &= \int_{\R \setminus \{x-\tau,x+\tau\}} \rho(t) \, F'(x-t) \, dt= \int_{\R \setminus \{-\tau,\tau\}} \rho(x-t) \, F'(t) \, dt  \\
&= -2 \int_{-\tau}^{\tau} \rho(x-t) \,dt = -2 \int_{x-\tau}^{x+\tau} \underbrace{\rho(t)}_{\geq 0} \, dt \geq -2 \int_{-\infty}^{\infty} \rho(t) \,dt = -2.
\end{split}
\end{align}
 This  proves Item~\eqref{item:f:0}. In addition, observe that \eqref{eq:f:der} ensures for all $x \in (0,\tau)$  that
\begin{align}
\begin{split}
f'(x)& = -2 \int_{x-\tau}^{x+\tau} \rho(t) \,dt = -2 \int_{x-\tau}^{0} \rho(t) \,dt -2 \int_{0}^{\varepsilon} \rho(t) \,dt\\
& = -2 \int_{x-\tau}^{0} \rho(t) \,dt -1 < -1.
\end{split}
\end{align}
This establishes Item~\eqref{item:f:-1}. Next note that \eqref{eq:f:int}  yields that for all $x \in (0, \tau -\varepsilon)$ it holds that 
\begin{align}
\begin{split}
f(x) &= \int_{-\varepsilon}^{\varepsilon} \rho(t) \, F(x-t)\, dt = \int_{-\varepsilon}^{\varepsilon} \rho(t) \, (2\tau-2(x-t)) \, dt \\ & =\int_{-\varepsilon}^{\varepsilon} \rho(t) \, (2\tau-2x+2t) \, dt \geq (2\tau-2x-2\varepsilon) \int_{-\varepsilon}^{\varepsilon} \rho(t) \, dt \\
&= (2\tau-2x-2\varepsilon).
\end{split}
\end{align}
Hence, we obtain that
\begin{align}
\begin{split}
\int_0^{\tau_1} |f(s)|^2 \, ds  & \geq  \int_0^{\tau -\varepsilon} |f(s)|^2 \, ds \geq \int_0^{\tau-\varepsilon} (2\tau-2s-2\varepsilon)^2 \, ds \\
&= 4 \int_0^{\tau-\varepsilon} (\tau-\varepsilon-s)^2 \, ds = 4 \int_0^{\tau-\varepsilon} s^2 \, ds  \\
&= \frac{4(\tau-\varepsilon)^3}{3} \geq \frac{4}{3} \left[ \tau - \tau(1- 2^{\nicefrac{-1}{3}}) \right]^3  \\
& = \frac{4\tau^3}{3} \left[1-1+2^{\nicefrac{-1}{3}}\right]^3= \frac{4\tau^3}{3} \cdot \frac{1}{2} = \frac{2\tau^3}{3}.
\end{split}
\end{align}
This demonstrates Item~\eqref{item:f:int}. The proof of Lemma~\ref{lemma:f} is thus completed.
\end{proof}

\subsection{Properties of the function appearing in the second component of the considered two-dimensional SDE}

The next result, Lemma~\ref{lemma:g}, establishes a few elementary (regularity) properties of the function $g\colon \R \to \R$ in \eqref{eq:setting:g} in Section~\ref{setting:fg}.

\begin{lemma}
\label{lemma:g}
Assume the setting in Section~\ref{setting:fg}. Then  
\begin{enumerate}[(i)]
\item \label{item:g:0} it holds that $g((-\infty,\tau_1]) = \{0\}$,
\item \label{item:g:4} it holds that $g([ \tau_2, \infty)) = \{4\}$, 
\item \label{item:g:mon} it holds that $g'(\R) \subseteq [0, \infty)$, 
\item \label{item:g:st:mon} it holds that $g'((\tau_1,\tau_2)) \subseteq (0, \infty)$, and
\item \label{item:g:b} it holds that $\sup_{x \in \R} |g(x)| < \infty$.
\end{enumerate}
\end{lemma}
\begin{proof}[Proof of Lemma~\ref{lemma:g}]
First, note that for all $x \in (-\infty,\tau_1]$ it holds that $h(x-\tau_1)=0$ and $h(\tau_2-x) >0$. This proves Item~\eqref{item:g:0}. Next observe that for all $x \in [\tau_2, \infty)$ it holds  that $h(\tau_2-x)=0$ and $h(x-\tau_1)>0$. This demonstrates that for all $x \in [\tau_2, \infty)$ it holds that
\begin{align}
g(x) =  \frac{4h(x-\tau_1)}{h(x-\tau_1)} = 4.
\end{align}
This proves Item~\eqref{item:g:4}. In the next step we note that the fact that $\forall \, x \in \R \colon h'(x) \geq 0$ ensures that for all $x \in \R$ it holds that 
\begin{align}
\begin{split}
g'(x) &= \frac{4h'(x-\tau_1)(h(x-\tau_1)+h(\tau_2-x))-4h(x-\tau_1)(h'(x-\tau_1)-h'(\tau_2-x))}{(h(x-\tau_1)+h(\tau_2-x))^2}\\
& = \frac{4h'(x-\tau_1)h(\tau_2-x)+4h(x-\tau_1)h'(\tau_2-x)}{(h(x-\tau_1)+h(\tau_2-x))^2}  \geq 0.
\end{split}
\end{align}
This proves Items~\eqref{item:g:mon}--\eqref{item:g:st:mon}. Item~\eqref{item:g:b} is an immediate consequence of Items~\eqref{item:g:0}--\eqref{item:g:mon}. The proof of Lemma~\ref{lemma:g} is thus completed.
\end{proof}

\subsection{A concrete example for the functions appearing in the considered two-dimensional SDE}
\label{sub:concrete}
Assume the setting in Section~\ref{setting:fg} and assume that
\begin{align}
T = \frac{3}{2}, \qquad \tau = \frac{3}{4}, \qquad \varepsilon = \frac{4 \min\{T-\tau, \tau(1- 2^{\nicefrac{-1}{3}})\}}{5}, \qquad \text{and} \qquad  \tau_2 = \tau_1 + \frac{4 (T-\tau_1)}{5}.
\end{align}
Observe that these hypotheses ensure that
\begin{align}
\begin{split}
\varepsilon = \frac{4\tau(1- 2^{\nicefrac{-1}{3}})}{5} =  \frac{3(1- 2^{\nicefrac{-1}{3}})}{5} \approx 0.1238, \qquad \tau_1 = \tau+\varepsilon \approx 0.8738,
\end{split}
\end{align} 
and 
\begin{align}
 \mu = \int_{-\varepsilon}^{\varepsilon} \exp\!\Big(\tfrac{-1}{(\varepsilon^2 - t^2)}\Big) \, dt \approx 1.2 \cdot 10^{-30}.
\end{align}
In Figure~\ref{fig:fg} we approximately plot $f(x)$ and $g(x)$ against $x \in [-\frac{85}{100},\frac{16}{10} ]$.

\begin{figure}
	\centering
	\includegraphics[width=.9\textwidth]{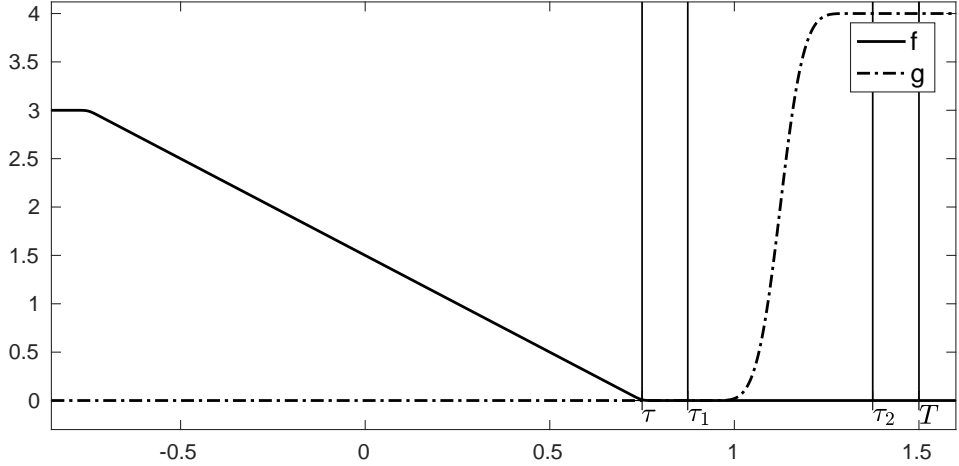}
	\caption{Plot of the functions $f$ and $g$ from Subsection~\ref{sub:concrete}}
	\label{fig:fg}
\end{figure}


\section{Lower bounds for strong approximation errors}

\subsection{Setting}
\label{setting:gen}

Let $T \in (0, \infty)$, $\tau \in (0, T)$, $\tau_1 \in (\tau,T)$, $\tau_2 \in (\tau_1,T)$, $\alpha \in [\frac{2\tau^3}{3}, \infty)$,  $f,g \in C^{\infty}(\R,\R)$  satisfy $\sup_{x \in \R} (|f(x)|+|g(x)|) < \infty$, $f((-\infty,\tau_1)) \subseteq (0,\infty)$, $f([\tau_1,\infty)) = \{0\}$, $f'(\R) \subseteq [-2,0]$, $f'((0,\tau)) \subseteq [-2, -1)$, $g((-\infty,\tau_1]) = \{0\}$, $g([ \tau_2, \infty)) = \{4\}$, $g'(\R) \subseteq [0, \infty)$, $ \alpha = \int_0^{\tau_1} |f(s)|^2 \, ds $,  let $(\Omega, \mathcal{F}, \P)$ be a probability space with a normal filtration $\mathbb{F}=(\mathbb{F}_t)_{t \in [0,T]}$,   let $W \colon  [0,T] \times \Omega \to \R$ be a standard $(\Omega, \mathcal{F}, \P, \mathbb{F})$-Brownian motion, 
and for every $\psi \in  C^{\infty}(\R,\R)$ let  $\Xpsione, \Xpsitwo \colon [0,T] \times \Omega \to \R$ be continuous $\mathbb{F}$-adapted  stochastic processes which satisfy for all $t \in [0,T]$ that $\P \big(\Xpsione_t= \int_0^t  f(\Xpsitwo_s) \, dW_s\big)=1$ and 
\begin{align}
\P\Big(\Xpsitwo_t = t + \textstyle \smallint_0^t \displaystyle  g(\Xpsitwo_s) [\cos(\psi(\Xpsione_s)) +1] \, ds\Big)=1.
\end{align}

\subsection{Comments to the setting}

The following result, Corollary~\ref{cor:fg} below, illustrates that there do indeed exist functions $f,g \colon \R \to \R$ which fulfill the hypotheses in Section~\ref{setting:gen}. Corollary~\ref{cor:fg} is an immediate consequence of Lemma~\ref{lemma:f} and Lemma~\ref{lemma:g} in Section~\ref{sec:existence}.

\begin{cor}
\label{cor:fg}
Let $T \in (0,\infty)$, $\tau \in (0,T)$. Then there exist $\tau_1 \in (\tau,T)$, $\tau_2 \in (\tau_1,T)$, $f,g \in C^{\infty}(\R,\R)$  which satisfy 
$\sup_{x \in \R} \!\big(|f(x)| +|g(x)|\big) < \infty$,
$f((-\infty,\tau_1)) \cup g'((\tau_1,\tau_2)) \subseteq (0,\infty)$,
$f([\tau_1,\infty)) = g((-\infty,\tau_1]) = \{0\}$,
$f'(\R) \subseteq [-2,0]$,
$f'((0,\tau)) \subseteq [-2, -1)$,
$g([ \tau_2, \infty)) = \{4\}$, and $  \int_0^{\tau_1} |f(s)|^2 \, ds \geq \frac{2\tau^3}{3}$.
\end{cor}

\subsection{Comparison results for a family of one-dimensional deterministic ordinary differential equations}

In this section we establish three elementary comparison results for a specific type of ordinary differential equations (cf., e.g., Exercise~1.7 in Tao \cite{taononlinear} for similar results) which we employ in the proof of Theorem~\ref{thm:intro} above.

\begin{lemma}
\label{lemma:monotone:gen}
Assume the setting in Section~\ref{setting:gen} and let $z = (z_t(a))_{t \in [\tau_1,T], a \in \R} = (z(t,a))_{t \in [\tau_1,T], a \in \R} \colon [\tau_1,T] \times \R \to \R$ be a continuous function which satisfies for all $t \in [\tau_1,T]$, $a \in \R$ that
\begin{align}
\label{eq:monotone}
z_t(a) = \tau_1 + \int_{\tau_1}^t \big[ 1+ g(z_s(a))(a+1)\big] \, ds.
\end{align}
Then it holds for all $a\in \R$, $b \in (-\infty,a]$, $t \in [\tau_1,T]$   that
\begin{align}
z_t(a) \geq z_t(b).
\end{align}
\end{lemma}
\begin{proof}[Proof of Lemma~\ref{lemma:monotone:gen}]
Throughout this proof let $y \colon [\tau_1,T] \times \R \to \R$ be the function which satisfies for all $t \in [\tau_1,T]$, $a \in \R$ that
\begin{align}
y(t,a) = \big(\tfrac{\partial  }{\partial a} z  \big)(t,a).  
\end{align}
Next note that \eqref{eq:monotone} ensures that for all $t \in [\tau_1,T]$, $a \in \R$ it holds that
\begin{align}
\big(\tfrac{\partial  }{\partial t} z  \big)(t,a) = 1+ g(z(t,a))(a+1).
\end{align}
This implies that for all $t \in [\tau_1,T]$, $a \in \R$ it holds that
\begin{align}
\begin{split}
\big(\tfrac{\partial  }{\partial t} y  \big)(t,a) &=  \big(\tfrac{\partial^2  }{\partial t \partial a} z  \big)(t,a) = \big(\tfrac{\partial^2  }{\partial a \partial t} z  \big)(t,a)\\
&=   g(z(t,a)) +  g'(z(t,a))(a+1) \big(\tfrac{\partial  }{\partial a} z  \big)(t,a) \\
&= g(z(t,a)) +  g'(z(t,a))(a+1) y(t,a).
\end{split}
\end{align} 
Therefore, we obtain that for all $t \in [\tau_1,T]$, $a \in \R$ it holds that
\begin{align}
\begin{split}
y(t,a) & = e^{\int_{\tau_1}^{t} g'(z(u,a))(a+1) \, du} \, y(\tau_1,a) + \int_{\tau_1}^{t} e^{\int_s^t g'(z(u,a))(a+1) \, du} \, g(z(s,a)) \, ds\\
& = \int_{\tau_1}^{t} e^{\int_s^t g'(z(u,a))(a+1) \, du} \, g(z(s,a)) \, ds \geq 0. 
\end{split}
\end{align}
Combining this with the fundamental theorem of calculus completes the proof of Lemma~\ref{lemma:monotone:gen}.
\end{proof}

\begin{lemma}
\label{lemma:z:diff:gen}
Assume the setting in Section~\ref{setting:gen} and let $z \colon [\tau_1,T] \times \R  \to \R$  be a continuous function which satisfies for all $t \in [\tau_1,T]$, $a \in \R$ that
\begin{align}
z_t(a) = \tau_1 + \int_{\tau_1}^t \big[ 1+ g(z_s(a))(a+1)\big] \, ds.
\end{align}
Then it holds for all $a\in [-1, \infty)$, $b \in [-1,a]$, $t \in [\tau_2,T]$   that
\begin{align}
z_t(a)-z_t(b) \geq 4(a-b )(t-\tau_2).
\end{align}
\end{lemma}
\begin{proof}[Proof of Lemma~\ref{lemma:z:diff:gen}]
First, note that  Lemma~\ref{lemma:monotone:gen} ensures that for all $a\in [-1, \infty)$, $b \in [-1,a]$, $t \in [\tau_1,T]$ it holds that  
\begin{align}
z_t(a) \geq z_t(b).
\end{align}
The fact that  $g$ is a  non-decreasing function hence ensures that for all $a\in [-1, \infty)$, $b \in [-1,a]$, $t \in [\tau_1,T]$ it holds that  
\begin{align}
\label{eq:z:mono:gen}
g(z_t(a))(a+1) \geq g(z_t(b))(a+1) \geq g(z_t(b))(b+1).
\end{align}
Moreover, observe that for all $t \in [\tau_2,T]$, $r \in [-1, \infty)$ it holds that
\begin{align}
\begin{split}
z_t (r) &= \tau_1 + \int_{\tau_1}^t \big[ 1+ g(z_s(r))(r+1)\big] \, ds \\
&\geq \tau_1 + \int_{\tau_1}^{\tau_2} \big[ 1+ g(z_s(r))(r+1)\big] \, ds \geq \tau_2.
\end{split}
\end{align}
This, \eqref{eq:z:mono:gen}, and the assumption that $g([\tau_2, \infty)) = \{4\}$ imply that for all $a\in [-1, \infty)$, $b \in [-1,a]$, $t \in [\tau_2,T]$ it holds that
\begin{align}
\begin{split}
z_t(a) - z_t(b) &= \int_{\tau_1}^t \big[ g(z_s(a))(a+1) - g(z_s(b))(b+1) \big] \, ds  \\
&\geq \int_{\tau_2}^t \big[ g(z_s(a))(a+1) - g(z_s(b))(b+1) \big] \, ds \\
&= \int_{\tau_2}^t  4(a+1) - 4(b+1) \, ds  = \int_{\tau_2}^t 4(a-b) \, ds \\
&= 4(a-b )(t-\tau_2).
\end{split}
\end{align}
The proof of Lemma~\ref{lemma:z:diff:gen} is thus completed.
\end{proof}

The next result, Corollary~\ref{cor:z:diff}, is an immediate consequence of Lemma~\ref{lemma:z:diff:gen} above.

\begin{cor}
\label{cor:z:diff}
Assume the setting in Section~\ref{setting:gen} and let $z \colon [\tau_1,T] \times \R  \to \R$  be a continuous function which satisfies for all $t \in [\tau_1,T]$, $a \in \R$ that
\begin{align}
z_t(a) = \tau_1 + \int_{\tau_1}^t \big[ 1+ g(z_s(a))(a+1)\big] \, ds.
\end{align}
Then it holds for all $a,b\in [-1, \infty)$  that
\begin{align}
|z_T(a)-z_T(b)| \geq 4(T-\tau_2)|a-b|.
\end{align}
\end{cor}

\subsection{On the explicit solution of a one-dimensional deterministic ordinary differential equation}

The second component of the two-dimensional SDE in Section~\ref{setting:gen} is partially employed to describe the time variable. It is the subject of the next two lemmas, Lemmas~\ref{lemma:x2:iden:gen} and~\ref{lemma:X:iden:gen}, to make this statement precise. Lemma~\ref{lemma:x2:iden:gen} is used in the proof of  Lemma~\ref{lemma:X:iden:gen}. Lemma~\ref{lemma:X:iden:gen}, in turn, is employed in the proof of Lemma~\ref{lemma:xtwo:z:gen} in Section~\ref{sec:solution} below.         

\begin{lemma}
\label{lemma:x2:iden:gen}
Let $T \in (0,\infty)$, $\tau_1 \in [0,T]$, $f,x \in C([0,T],[0,\infty))$, $g \in C(\R, [0,\infty))$   satisfy for all $ t \in [0,T]$ that $g((-\infty,\tau_1]) = \{0\}$ and 
\begin{align}
x_t = t + \int_0^t g(x_s) f(s) \, ds = \int_0^t \big[1+ g(x_s) f(s)\big] \, ds.
\end{align}
Then it holds for all $t \in [0,\tau_1]$  that $x_t = t$.
\end{lemma}
\begin{proof}[Proof of Lemma~\ref{lemma:x2:iden:gen}]
Throughout this proof let $\mu \in [0,T]$ be the real number given by
\begin{align}
\label{eq:tau:gen}
\mu = \inf \!\left( \{t \in [0,T] \colon x_t \geq \tau_1\} \cup \{T\}\right).
\end{align}
Observe that the fact that
\begin{align}
\forall \, t \in [0,T] \colon x_t = t + \int_0^t g(x_s) f(s) \, ds \geq t
\end{align}
ensures that 
\begin{align}
\label{eq:nonempty}
\{t \in [0,T] \colon x_t \geq \tau_1 \} \supseteq [\tau_1,T] \neq \emptyset.
\end{align}
Next note that the fact that $x_0=0$ assures that  for all $t \in [0,\mu]$ it holds that $x_t \leq \tau_1$. This and the assumption that $g((-\infty,\tau_1]) = \{0\}$ ensure that for all $t \in [0,\mu]$ it holds that
\begin{align}
\label{eq:tau:2:gen}
\tau_1 \geq x_t = t + \int_0^t g(x_s) f(s) \, ds = t.
\end{align}
In the next step we observe that \eqref{eq:tau:gen}  and \eqref{eq:nonempty} imply that $x_{\mu} \geq \tau_1$. Combining this with \eqref{eq:tau:2:gen} yields that
\begin{align}
\tau_1 \geq \mu = x_{\mu} \geq \tau_1.
\end{align}
This proves that $\mu=\tau_1$. Combining this and \eqref{eq:tau:2:gen} completes the proof of Lemma~\ref{lemma:x2:iden:gen}.
\end{proof}

\subsection{On the explicit solution of a two-dimensional SDE}
\label{sec:solution}

In this section we derive in Item~\eqref{item:tau1:X} of Lemma~\ref{lemma:X:iden:gen} and in Lemma~\ref{lemma:xtwo:z:gen} below an explicit representation of the solution of the SDE from Section~\ref{setting:gen}. This explicit representation is then employed in our error analysis in Section~\ref{sec:explicit} below.

\begin{lemma}
\label{lemma:X:iden:gen}
Assume the setting in Section~\ref{setting:gen} and let $\psi \in  C^{\infty}(\R,\R)$. Then
\begin{enumerate}[(i)]
\item \label{item:0:X} it holds for all $t \in [0,\tau_1]$  that $\P(\Xpsitwo_t =t)=1$,
\item \label{item:tau1:f} it holds for all $t \in [\tau_1,T]$  that $\P(f(\Xpsitwo_t)=0)=1$, and 
\item \label{item:tau1:X} it holds for all $t \in [\tau_1,T]$ that $\P(\Xpsione_t=\Xpsione_{\tau_1}= \int_0^{\tau_1}  f(s) \, dW_s)=1$.
\end{enumerate}
\end{lemma}
\begin{proof}[Proof of Lemma~\ref{lemma:X:iden:gen}]
First, note that Lemma~\ref{lemma:x2:iden:gen} proves that for all $t\in [0,\tau_1]$ it holds that $\P(\Xpsitwo_t =t)=1$. This establishes Item~\eqref{item:0:X}. Next note that  the fact that $g \geq 0$ ensures that for all $t \in [\tau_1,T]$ it holds that 
\begin{align}
\P(\Xpsitwo_t \geq \tau_1)=1.
\end{align}
The assumption that $f([\tau_1, \infty))= \{0\}$ hence proves Item~\eqref{item:tau1:f}. Moreover, observe that Item~\eqref{item:0:X} and Item~\eqref{item:tau1:f} imply that for all $t \in [\tau_1,T]$ it holds  $\P$-a.s.~that
\begin{align}
\Xpsione_t= \int_0^{\tau_1}  f(\Xpsitwo_s) \, dW_s + \int_{\tau_1}^t  f(\Xpsitwo_s) \, dW_s = \int_0^{\tau_1}  f(s) \, dW_s. 
\end{align}
This establishes Item~\eqref{item:tau1:X}. The proof of Lemma~\ref{lemma:X:iden:gen} is thus completed.
\end{proof}

\begin{lemma}
\label{lemma:xtwo:z:gen}
Assume the setting in Section~\ref{setting:gen}, let $\psi \in  C^{\infty}(\R,\R)$, and let $z \colon [\tau_1,T] \times \R  \to \R$ be a continuous function which satisfies for all $t \in [\tau_1,T]$, $a \in \R$  that
\begin{align}
z_t(a) = \tau_1 + \int_{\tau_1}^t \big[ 1+ g(z_s(a))(a+1)\big] \, ds.
\end{align}
Then it holds for all $t \in [\tau_1,T]$ that
\begin{align}
\P \Big(\Xpsitwo_t =z_t\big(\!\cos(\psi(\Xpsione_{\tau_1}))\big)\Big)=1.
\end{align}
\end{lemma}
\begin{proof}[Proof of Lemma~\ref{lemma:xtwo:z:gen}]
First, note that for all $t \in [\tau_1,T]$ it holds that
\begin{align}
\begin{split}
 1&= \P\Big(\Xpsitwo_t= \smallint\nolimits_0^t 1+ g(\Xpsitwo_s) [\cos(\psi(\Xpsione_s)) +1] \, ds\Big)\\
& = \P\Big(\Xpsitwo_t= \smallint\nolimits_0^{\tau_1} 1+ g(\Xpsitwo_s) [\cos(\psi(\Xpsione_s)) +1] \, ds\\
& \quad + \smallint\nolimits_{\tau_1}^t 1+ g(\Xpsitwo_s) [\cos(\psi(\Xpsione_s)) +1] \, ds \Big)\\
& = \P\Big(\Xpsitwo_t= \Xpsitwo_{\tau_1} + \smallint\nolimits_{\tau_1}^t 1+ g(\Xpsitwo_s) [\cos(\psi(\Xpsione_s)) +1] \, ds \Big).
\end{split}
\end{align}
Items~\eqref{item:0:X} and \eqref{item:tau1:X} of Lemma~\ref{lemma:X:iden:gen} hence prove that for all $t \in [\tau_1,T]$ it holds that
\begin{align}
\P \Big(\Xpsitwo_t =\tau_1+ \smallint\nolimits_{\tau_1}^t 1+ g(\Xpsitwo_s) [\cos(\psi(\Xpsione_{\tau_1})) +1] \, ds \Big)=1.
\end{align}
The fact that $\Xpsitwo$ is a continuous stochastic process therefore ensures that 
\begin{align}
\P \Big( \forall \, t \in [\tau_1,T] \colon \Xpsitwo_t =\tau_1+ \smallint\nolimits_{\tau_1}^t 1+ g(\Xpsitwo_s) [\cos(\psi(\Xpsione_{\tau_1})) +1] \, ds \Big)=1.
\end{align}
This completes the proof of Lemma~\ref{lemma:xtwo:z:gen}.
\end{proof}

\subsection{Lower and upper bounds for the variances of some Gaussian distributed random variables}

\begin{lemma}
\label{lemma:BM:gen}
Assume the setting in Section~\ref{setting:gen} and let $a\in [0, \tau)$, $b \in (a, \tau]$, let $\bar{W}, B \colon [a, b] \times \Omega \to \R$ and $\tilde{W} \colon ([0,a] \cup [b,T]) \times \Omega \to \R$ be stochastic processes, let $Y_1,  Y_2 \colon \Omega \to \R$ be random variables, and assume  for all $s \in [a,b]$, $t \in ([0,a]\cup [b,T])$ that  
\begin{align}
\tilde{W}_t= W_t, \quad \bar{W}_s = \frac{(s-a)}{(b-a)} \cdot W_b + \frac{(b-s)}{(b-a)} \cdot W_a, \quad B_s = W_s - \bar{W}_s,
\end{align}
\begin{align}
\P \Big(Y_1 = \smallint\nolimits_{0}^a f(s) \, dW_s + \smallint\nolimits_{b}^{\tau_1} f(s) \, dW_s + \smallint\nolimits_a^b f(s) \, d\bar{W}_s \Big) =1,
\end{align}
and
\begin{align}
\P \Big( Y_2 = \smallint\nolimits_a^b f(s)\, dW_s - \smallint\nolimits_a^b f(s)\, d\bar{W}_s \Big)=1.
\end{align}
Then 
\begin{enumerate}[(i)]
\item \label{item:BM:indep} it holds that $\Omega \in \omega \mapsto (\tilde{W}_t(\omega))_{t \in [0,a] \cup [b,T]} \in C([0,a] \cup [b,T],\R)$ and  $\Omega \in \omega \mapsto (B_t(\omega))_{t \in [a,b]} \in C([a,b],\R)$  are independent on $(\Omega, \mathcal{F}, \P)$,
\item \label{item:Ito} it holds for all $t_1, t_2 \in [0,T]$ with $t_1 \leq t_2$ that 
\begin{align}
 \P \Big( \smallint\nolimits_{t_1}^{t_2} f(s) \, dW_s = f(t_2) W_{t_2} - f(t_1) W_{t_1} - \smallint\nolimits_{t_1}^{t_2} f'(s) W_s \, ds \Big) =1,
\end{align}
\item \label{item:Y2}  it holds that 
\begin{align}
\P \Big( Y_2  = - \smallint\nolimits_{a}^{b} f'(s) B_s \, ds \Big) =1,
\end{align}
\item \label{item:BM:1} it holds that $\frac{\alpha}{2} \leq \E \big[|Y_1|^2\big] \leq \alpha$, and
\item \label{item:BM:2} it holds that $\frac{(b-a)^3}{12} \leq \E \big[|Y_2|^2\big]  \leq \frac{(b-a)^3}{3}$.
\end{enumerate}
\end{lemma}
\begin{proof}[Proof of Lemma~\ref{lemma:BM:gen}]
First, note that for all $n \in \N$, $t_1,\ldots,t_n \in [0,T]$ it holds that 
\begin{align}
\label{eq:n:W}
\Omega \ni \omega \mapsto (W_{t_1}(\omega),\ldots,W_{t_n}(\omega)) \in \R^n
\end{align}
is Gaussian distributed. 
Next note that for all $s \in [a,b]$, $u \in [0,a] \cup [b,T]$ it holds that
\begin{align}
\label{eq:B:W}
\begin{split}
\E[B_s \tilde{W}_u] &= \E\! \left[\left( W_s - \frac{(s-a)}{(b-a)} \cdot W_b - \frac{(b-s)}{(b-a)} \cdot W_a \right)W_u\right]\\
&= \min\{s,u\} - \frac{(s-a) \min\{b,u\}}{(b-a)}  - \frac{(b-s)\min\{a,u\}}{(b-a)}\\
&= \frac{(b-a)\min\{s,u\} - (s-a) \min\{b,u\} -(b-s)\min\{a,u\} }{(b-a)}\\
& =  \begin{cases}
\frac{(b-a)u - (s-a)u -(b-s)u }{(b-a)} = \frac{bu-au - su+au -bu+su }{(b-a)} &\colon \quad u \leq a \\ 
\frac{(b-a)s- (s-a) b -(b-s)a }{(b-a)} = \frac{bs-as- sb+a b -ba+sa }{(b-a)} &\colon \quad u \geq b \\
\end{cases} \\
& =  0 .
\end{split}
\end{align}
Combining this with \eqref{eq:n:W} ensures that for all $n,m \in \N$, $t_1, \ldots, t_n \in [0,a] \cup [b,T]$, $s_1, \ldots, s_m \in [a,b]$, $\mathbb{W}_1, \ldots, \mathbb{W}_n$, $\mathbb{B}_1, \ldots, \mathbb{B}_m \in \mathcal{B}(\R)$ it holds that
\begin{align}
\begin{split}
& \P \Big( \Big\{ \big(\tilde{W}_{t_1}, \ldots, \tilde{W}_{t_n}\big) \in \mathbb{W}_1 \times \ldots \times \mathbb{W}_n \Big\} \cap \Big\{ \big(B_{s_1}, \ldots, B_{s_m}\big) \in \mathbb{B}_1 \times \ldots \times \mathbb{B}_m \Big\} \Big)\\
& = \P \Big( \big(\tilde{W}_{t_1}, \ldots, \tilde{W}_{t_n}\big) \in \mathbb{W}_1 \times \ldots \times \mathbb{W}_n \Big)  \cdot \P \Big( \big(B_{s_1}, \ldots, B_{s_m}\big) \in \mathbb{B}_1 \times \ldots \times \mathbb{B}_m  \Big).
\end{split}
\end{align} 
This, the fact that 
\begin{align}
\mathcal{B}\big(C([a,b], \R)\big) = \mathcal{B}(\R)^{\otimes [a,b]} \Cap C([a,b],\R),
\end{align}
and the fact that 
\begin{align}
\mathcal{B}\big(C([0,a]\cup [b,T], \R)\big) = \mathcal{B}(\R)^{\otimes [0,a]\cup [b,T]} \Cap C([0,a]\cup [b,T],\R)
\end{align}
establish Item~\eqref{item:BM:indep}. Moreover, note that \eqref{eq:B:W} proves that for all $s,u \in [a,b]$ it holds that
\begin{align}
\label{eq:W:cor}
\begin{split}
\E[B_s \bar{W}_u] &= \E\! \left[B_s\!\left( \frac{(u-a)}{(b-a)} \cdot W_b + \frac{(b-u)}{(b-a)} \cdot W_a\right)\!\right]\\
&= \E\! \left[B_s\!\left( \frac{(u-a)}{(b-a)} \cdot \tilde{W}_b + \frac{(b-u)}{(b-a)} \cdot \tilde{W}_a\right)\!\right]\\
&  = \frac{(u-a)}{(b-a)} \cdot \E[B_s \tilde{W}_b] + \frac{(b-u)}{(b-a)} \cdot \E[B_s \tilde{W}_a] = 0.
\end{split}
\end{align}
Hence, we obtain  that for all $s, u \in [a,b]$ it holds that
\begin{align}
\begin{split}
&\E[B_s B_u] = \E[B_s(W_u-\bar{W}_u)]  = \E[B_s W_u] \\
&= \E\! \left[\left( W_s - \frac{(s-a)}{(b-a)} \cdot W_b - \frac{(b-s)}{(b-a)} \cdot W_a \right)W_u\right]\\
&= \min\{s,u\} - \frac{u(s-a)}{(b-a)}  - \frac{a(b-s)}{(b-a)} \\
&= \frac{(b-a)\min\{s,u\} -us+au - ab+as}{(b-a)}\\
&= \frac{b\min\{s,u\} - \max\{s,u\}\min\{s,u\}+a(u+s - \min\{s,u\}) -ab}{(b-a)}\\
&= \frac{(b - \max\{s,u\})\min\{s,u\}-a(b - \max\{s,u\})}{(b-a)}\\
&=  \frac{(b-\max\{s,u\})(\min\{s,u\}-a)}{(b-a)} .
\end{split}
\end{align}
Moreover, observe that It\^o's formula ensures that for all $t_1, t_2 \in [0,T]$ with $t_1 \leq t_2$ it holds $\P$-a.s.~that 
\begin{align}
f(t_2) W_{t_2} =f(t_1) W_{t_1} + \int_{t_1}^{t_2} f'(s) W_s \, ds+ \int_{t_1}^{t_2} f(s) \, dW_s.
\end{align}
Hence, we  obtain that for all $t_1, t_2 \in [0,T]$ with $t_1 \leq t_2$ it holds $\P$-a.s.~that 
\begin{align}
\label{eq:Ito}
\int_{t_1}^{t_2} f(s) \, dW_s = f(t_2) W_{t_2} - f(t_1) W_{t_1} - \int_{t_1}^{t_2} f'(s) W_s \, ds.
\end{align}
This establishes Item~\eqref{item:Ito}. In addition, note that \eqref{eq:Ito} assures that it holds $\P$-a.s.~that 
\begin{align}
\label{eq:Y2:Ito}
\begin{split}
Y_2 & = f(b) W_{b} - f(a) W_{a} - \int_{a}^{b} f'(s) W_s \, ds - \int_a^b f(s)\, d\bar{W}_s \\
& = f(b) W_{b} - f(a) W_{a} - \int_{a}^{b} f'(s) W_s \, ds - \frac{W_b}{(b-a)} \int_a^b
f(s) \, ds \\
& \quad + \frac{W_a}{(b-a)} \int_a^b f(s) \, ds.
\end{split}
\end{align}
Furthermore, note that integration by parts shows that
\begin{align}
\label{eq:int:fb}
\begin{split}
\int_a^b f(s) \, ds &= \int_a^b f(s)(s-a)^0 \, ds = \big[f(s)(s-a)\big]_{s=a}^{s=b} - \int_a^b f'(s)(s-a) \, ds \\
& = f(b)(b-a)  -\int_a^b f'(s)(s-a) \, ds
\end{split}
\end{align}
and 
\begin{align}
\label{eq:int:fa}
\int_a^b f(s) \, ds &= \int_a^b f(s)(b-s)^0 \, ds = - \big[f(s)(b-s)\big]_{s=a}^{s=b} + \int_a^b f'(s)(b-s) \, ds \nonumber \\
& = f(a)(b-a)  + \int_a^b f'(s)(b-s) \, ds.
\end{align}
Putting \eqref{eq:int:fb} and \eqref{eq:int:fa} into \eqref{eq:Y2:Ito} shows that it holds $\P$-a.s.~that 
\begin{align}
\begin{split}
Y_2 & =  - \int_{a}^{b} f'(s) W_s \, ds + \int_a^b f'(s) \! \left[\frac{(s-a)}{(b-a)} \cdot W_b \right] ds\\
& \quad   + \int_a^b f'(s)\! \left[\frac{(b-s)}{(b-a)} \cdot W_a \right] ds\\
& = - \int_{a}^{b} f'(s) W_s \, ds +  \int_{a}^{b} f'(s) \bar{W}_s \, ds = - \int_{a}^{b} f'(s) [W_s - \bar{W}_s ]\, ds\\
& = - \int_{a}^{b} f'(s) B_s \, ds.
\end{split}
\end{align}
This establishes Item~\eqref{item:Y2}. Next note that Item~\eqref{item:Y2} proves that
\begin{align}
\label{eq:BM:gen}
\begin{split}
\E\big[|Y_2|^2\big] &= \E \! \left[\left|\int_a^b f'(s)  B_s \, ds\right|^2\right] = \int_a^b \int_a^b f'(s) f'(u) \, \E[B_s B_u] \,ds \,  du \\
& = \int_a^b \int_a^b f'(s) f'(u) \! \left[\frac{(b-\max\{s,u\})(\min\{s,u\}-a)}{(b-a)} \right]  ds \, du. 
\end{split}
\end{align}
Moreover, observe that
\begin{align}
\begin{split}
& \int_a^b \int_a^b  \frac{(b-\max\{s,u\})(\min\{s,u\}-a)}{(b-a)} \, ds \, du \\
& =  \int_a^b \int_a^u  \frac{(b-u)(s-a)}{(b-a)} \, ds \, du + \int_a^b \int_u^b  \frac{(b-s)(u-a)}{(b-a)} \, ds \, du\\
& = \int_a^b \frac{(b-u)}{(b-a)} \left[ \int_0^{u-a} s   \, ds \right] du + \int_a^b  \frac{(u-a)}{(b-a)} \left[ \int_0^{b-u}  s \, ds \right] du\\
& = \int_a^b   \frac{(b-u)(u-a)^2}{2(b-a)}  \, du + \int_a^b  \frac{(b-u)^2(u-a)}{2(b-a)} \, du\\ 
& = \int_a^b   \frac{(b-u)(u-a)}{2}  \, du = \int_0^{b-a}   \frac{(b-a-u)u}{2}  \, du \\
& = \int_0^{b-a}   \frac{(b-a)u}{2}  \, du - \int_0^{b-a} \frac{u^2}{2} \, du= \frac{(b-a)}{2}\cdot \frac{(b-a)^2}{2} -\frac{(b-a)^3}{6} \\
&=\left[\frac{1}{4} - \frac{1}{6}\right]\! (b-a)^3= \frac{(b-a)^3}{12}.
\end{split}
\end{align}
The assumption  that $f'((0, \tau)) \subseteq [-2,-1)$ and \eqref{eq:BM:gen} hence ensure that
\begin{align}
\frac{(b-a)^3}{12} \leq \E\big[|Y_2|^2\big] \leq \frac{(b-a)^3}{3}.
\end{align}
This establishes Item~\eqref{item:BM:2}. Next note that Item~\eqref{item:BM:indep} proves that the random variables $Y_1$ and $Y_2$ are independent. It\^o's isometry hence yields that
\begin{align}
\begin{split}
\E\big[|Y_1|^2\big] &= \E\big[|Y_1+Y_2|^2\big] - \E\big[|Y_2|^2\big] - 2\,\E\big[Y_1 Y_2\big] \\
&= \E \left[\left|\int_0^{\tau_1} f(s) \, dW_s \right|^2\right] -\E\big[|Y_2|^2\big] \\
&= \int_0^{\tau_1} |f(s)|^2 \, ds -\E\big[|Y_2|^2\big] = \alpha - \E\big[|Y_2|^2\big] \leq \alpha.
\end{split}
\end{align}
The assumption that $\alpha \geq \frac{2\tau^3}{3}$, the fact that $(b-a) \in (0,\tau]$, and Item~\eqref{item:BM:2} therefore ensure that
\begin{align}
\alpha \geq \E\big[|Y_1|^2\big] \geq \alpha - \frac{(b-a)^3}{3} \geq \alpha - \frac{\tau^3}{3}  \geq \frac{\alpha}{2}.
\end{align}
This  establishes Item~\eqref{item:BM:1}. The proof of Lemma~\ref{lemma:BM:gen} is thus completed. 
\end{proof}

\subsection{Explicit lower bounds for strong approximation errors for two-dimensional SDEs}
\label{sec:explicit}

The main result of this section, Lemma~\ref{lemma:lower:1:gen} below, establishes an explicit lower error bound for a large class of strong approximations of the solution process of the SDE in Section~\ref{setting:gen}. The proof of Lemma~\ref{lemma:lower:1:gen} uses the following two auxiliary lemmas (Lemmas~\ref{lemma:P:gen} and \ref{lem:sin} below). Lemma~\ref{lemma:P:gen}  is proved as Lemma~4.1 in \cite{jentzen2016slow}. 
\begin{lemma}
\label{lemma:P:gen}
Let $(\Omega, \mathcal{F}, \P)$ be a probability space, let $(S_1, \mathcal{S}_1)$ and $(S_2, \mathcal{S}_2)$ be measurable spaces, and let $X_1 \colon \Omega \to S_1$ and $X_2, X_2', X_2'' \colon \Omega \to S_2$ be random variables such that
\begin{align}
\P_{(X_1,X_2)} = \P_{(X_1,X_2')} = \P_{(X_1, X_2'')}.
\end{align}	
Then it holds for all measurable functions $\Phi \colon S_1 \times S_2 \to \R$ and $\varphi \colon S_1 \to \R$ that
\begin{align}
\E\big[|\Phi(X_1,X_2) - \varphi(X_1)|\big] \geq \tfrac{1}{2} \, \E\big[|\Phi(X_1,X_2') - \Phi(X_1,X_2'')|\big].
\end{align}
\end{lemma}

\begin{lemma}
	\label{lem:sin}
Let $c \in \R$, $\beta \in (0, 1)$ and let $\lambda \colon \mathcal{B}(\R) \to [0, \infty]$ the Lebesgue-Borel measure on $\R$. Then 
\begin{align}
\label{eq:lem:sin}
\lambda \big(\big\{ x \in [c-1, c+1] \colon \big|\!\sin\!\big(\tfrac{x-c}{\beta}\big)| \geq \tfrac{1}{2}\big\}\big) \geq \tfrac{1}{2}.
\end{align}
\end{lemma}
\begin{proof}[Proof of Lemma~\ref{lem:sin}]
Throughout this proof let $A \subseteq \R$ be the set given by 
\begin{align}
A = \big\{ x \in [c-1, c+1] \colon \big|\!\sin\!\big(\tfrac{x-c}{\beta}\big)| \geq \tfrac{1}{2}\big\}
\end{align}
and let $m \in \Z$ be the integer number which satisfies that
\begin{align}
\label{eq:beta:m}
\beta \pi \big(m -1 +\tfrac{1}{6}\big) < -1 \qquad \text{and} \qquad \beta \pi \big(m  +\tfrac{1}{6}\big) \geq -1.
\end{align}
Observe that the fact that $\forall \, k \in \Z \colon \sin\!\big(\frac{\pi}{6}+ k \pi\big) = \sin\!\big(\frac{5\pi}{6}+ k \pi\big) = (-1)^k \cdot \frac{1}{2}$ ensures that
\begin{align}
\big\{ y \in \R \colon  \left|\sin(y)\right|  \geq \tfrac{1}{2} \big\} = \cup_{k \in \Z} \big[\tfrac{\pi}{6}+ k \pi, \tfrac{5\pi}{6}+ k \pi\big].
\end{align}
Hence, we obtain that
\begin{align}
\label{eq:sin:A}
\begin{split}
A &= [c-1,c+1] \bigcap \bigg( \cup_{k \in \Z} \left\{ x \in \R \colon \big( \tfrac{x-c}{\beta} \big) \in \big[\tfrac{\pi}{6} + k \pi, \tfrac{5\pi}{6} + k \pi\big]\right\}\!\bigg)\\
& = [c-1,c+1] \bigcap \Big( \cup_{k \in \Z}  \big[c+ \beta\big(\tfrac{\pi}{6} + k \pi\big), c+ \beta \big(\tfrac{5\pi}{6} + k \pi\big)\big] \Big)\\
& \supseteq [c-1,c+1] \bigcap \Big( \cup_{k = m-1}^{\infty}  \big[c+ \beta\big(\tfrac{\pi}{6} + k \pi\big), c+ \beta \big(\tfrac{5\pi}{6} + k \pi\big)\big] \Big).
\end{split}
\end{align}
Next note that \eqref{eq:beta:m} and the assumption that $\beta \in (0,1)$ ensure that $m \leq 0$. To prove \eqref{eq:lem:sin}, we distinguish between two cases.
In the first case we assume that $m=0$. We observe that \eqref{eq:beta:m} then yields that 
\begin{align}
\beta > \tfrac{6}{5 \pi}.
\end{align}
This and the fact that $\beta \in (0,1)$ prove that 
\begin{align}
c+  \tfrac{5  \beta \pi}{6} > c+1,
\end{align}
\begin{align}
 c- \tfrac{\beta \pi}{6} >  c-   \tfrac{\pi}{6} > c-1,
\end{align}
and 
\begin{align}
	c+ \tfrac{\beta \pi}{6} <  c+   \tfrac{\pi}{6} < c+1.
\end{align}
Combining this, \eqref{eq:sin:A}, and \eqref{eq:beta:m} ensures that 
\begin{align}
\begin{split}
A &\supseteq [c-1,c+1] \bigcap \Bigg( \bigcup_{k = -1}^{0}  \big[c+ \beta\big(\tfrac{\pi}{6} + k\pi \big), c+ \beta \big(\tfrac{5\pi}{6} + k  \pi \big)\big] \Bigg)\\
& = [c-1,c+1] \bigcap \left( \big[c - \tfrac{5 \beta \pi}{6} , c - \tfrac{ \beta\pi}{6} \big] \bigcup \big[c+ \tfrac{\beta \pi}{6} , c+ \tfrac{5 \beta \pi}{6} \big] \right) \\
& = \big[c -1 , c - \tfrac{ \beta\pi}{6} \big] \bigcup \big[c+ \tfrac{\beta \pi}{6} , c+ 1\big].
\end{split}
\end{align}
This implies that
\begin{align}
\lambda(A) \geq 2 \big(1- \tfrac{\beta \pi }{6}\big) > 2 - \tfrac{\pi}{3} > \tfrac{1}{2}.
\end{align}
This finishes the proof of \eqref{eq:lem:sin} in the case $m=0$. In the second case we assume that $m \leq -1$. Note that \eqref{eq:beta:m} proves that
\begin{align}
\beta \big(\tfrac{5\pi}{6} + \pi (-m-1)\big) = \beta \pi \big(-m - \tfrac{1}{6}\big) = -\beta \pi \big(m  +\tfrac{1}{6}\big) \leq 1.
\end{align}
This and again \eqref{eq:beta:m} ensure for all $k \in [m, -m-1] \cap \Z$  that 
\begin{align}
\big[c+ \beta\big(\tfrac{\pi}{6} + k \pi \big), c+ \beta \big(\tfrac{5\pi}{6} + k \pi \big)\big] \subseteq [c-1,c+1].
\end{align}
Combining \eqref{eq:sin:A} and \eqref{eq:beta:m} hence demonstrates that 
\begin{align}
\begin{split}
\lambda(A) &\geq \lambda\! \left( \bigcup_{k = m}^{-m-1}  \big[c+ \beta\big(\tfrac{\pi}{6} +k \pi \big), c+ \beta \big(\tfrac{5\pi}{6} + k \pi \big)\big] \right)\\
& = \sum_{k=m}^{-m-1} \lambda\! \left( \big[c+ \beta\big(\tfrac{\pi}{6} +k \pi \big), c+ \beta \big(\tfrac{5\pi}{6} +k \pi \big)\big]\right)\\
& = -2m \cdot \tfrac{2 \beta \pi}{3} > - \tfrac{4m}{3} \cdot \tfrac{1}{\nicefrac{5}{6} -m} = - \tfrac{8m}{5-6m} > \tfrac{1}{2}.
\end{split}
\end{align}
This finishes the proof of  \eqref{eq:lem:sin} in the case $m \leq -1$. The proof of Lemma~\ref{lem:sin} is thus completed.
\end{proof}

\begin{lemma}
\label{lemma:lower:1:gen}
Assume the setting in Section~\ref{setting:gen}, let $a\in [0, \tau)$, $b \in (a, \tau]$,  $c \in [2, \infty)$, $\psi \in  C^{\infty}(\R,\R)$, let $u \colon C([0,a] \cup [b,T], \R) \to \R$ be a measurable function, and assume for all $x \in [c-2, c+2]$ that  $\psi(x) = \frac{T^{3/2}}{(b-a)^{3/2}} \cdot (x-c)$. Then 
\begin{align}
\begin{split}
&\E \! \left[\big| \Xpsitwo_T - u((W_s)_{s \in [0,a] \cup [b,T]})\big|\right] \\
&\geq \frac{\sqrt{3}(T-\tau_2)}{\pi\sqrt{T^3 \alpha}} \left[\int_{c+\nicefrac{1}{2}}^{c+1}  e^{-\frac{x^2}{\alpha}} \, dx \right] \!\left[ \int_0^1  \left|\sin(y)\right|  e^{-\frac{6y^2}{T^3 }}\,  dy \right]
> 0.
\end{split}
\end{align}
\end{lemma}
\begin{proof}[Proof of Lemma~\ref{lemma:lower:1:gen}]
Throughout this proof let $A \subseteq \R$ be the set given by 
\begin{align}
A = \big\{x \in [c-1,c+1] \colon \left|\sin(\psi(x))\right| \geq \tfrac{1}{2}\big\},
\end{align}
let  $\bar{W}, B \colon [a, b] \times \Omega \to \R$ and $\tilde{W} \colon ([0,a] \cup [b,T]) \times \Omega \to \R$ be the stochastic processes with continuous sample paths which satisfy for all $s \in [a,b]$, $t \in ([0,a]\cup [b,T])$ that 
\begin{align}
\bar{W}_s& = \frac{(s-a)}{(b-a)} \cdot W_b + \frac{(b-s)}{(b-a)} \cdot W_a, \quad B_s = W_s - \bar{W}_s, \quad \text{and} \quad  \tilde{W}_t= W_t,
\end{align}
 let $Y_1,  Y_2 \colon \Omega \to \R$ be random variables which satisfy 
\begin{align}
\label{eq:lem:Y1}
\P \Big(Y_1 = \smallint\nolimits_{0}^a f(s) \, dW_s + \smallint\nolimits_{b}^{\tau_1} f(s) \, dW_s + \smallint\nolimits_a^b f(s) \, d\bar{W}_s \Big) =1,
\end{align}
and
\begin{align}
\P \Big( Y_2 = \smallint\nolimits_a^b f(s)\, dW_s - \smallint\nolimits_a^b f(s)\, d\bar{W}_s \Big)=1,
\end{align} 
let $z \colon [\tau_1,T] \times \R  \to \R$ be a continuous function which satisfies for all $ t \in [\tau_1,T]$, $a \in \R$ that 
\begin{align}
\label{eq:z}
z_t(a) = \tau_1 + \int_{\tau_1}^t \big[ 1+ g(z_s(a))(a+1)\big] \, ds,
\end{align}
let $\sigma_1, \sigma_2, \varepsilon, \beta \in (0,\infty)$ be the real numbers given by 
\begin{align}
\sigma_1 = \E\big[|Y_1|^2\big], \quad \sigma_2 = \E\big[|Y_2|^2\big], \quad \varepsilon = b-a, \quad \text{and} \quad \beta = \frac{\varepsilon^3}{T^3},
\end{align}
and for every $x \in \R$, $y \in (0,\infty)$ let $\mathcal{N}_{x,y} \colon \mathcal{B}(\R) \to [0,\infty]$ be the function which satisfies for all $B \in \mathcal{B}(\R)$ that
\begin{align}
\mathcal{N}_{x,y}(B) = \int_{B} \frac{1}{\sqrt{2\pi y}} \,\, e^{-\frac{(r-x)^2}{2y}} \, dr.
\end{align}
Next note that Item~\eqref{item:tau1:X} in Lemma~\ref{lemma:X:iden:gen} proves that for all $t \in [\tau_1,T]$ it holds $\P$-a.s.~that 
\begin{align}
\begin{split}
\Xpsione_t  &= \int_0^{\tau_1}  f(s) \, dW_s = \int_0^a  f(s) \, dW_s + \int_b^{\tau_1}  f(s) \, dW_s + \int_a^b  f(s) \, dW_s\\
& = \left[ \int_0^a  f(s) \, dW_s + \int_b^{\tau_1}  f(s) \, dW_s  + \int_a^b f(s) \, d\bar{W}_s \right] \\
& \quad + \left[\int_a^b  f(s) \, dW_s - \int_a^b  f(s) \, d\bar{W}_s\right] = Y_1 + Y_2.
\end{split}
\end{align}
This together with Lemma~\ref{lemma:xtwo:z:gen} ensures that
\begin{align}
\label{eq:zT}
\P\Big(\Xpsitwo_T= z_T\big(\!\cos(\psi(Y_1 +Y_2))\big)\Big)=1.
\end{align}
 Moreover, observe that  Items~\eqref{item:Ito} and \eqref{item:Y2} of Lemma~\ref{lemma:BM:gen} show that for all $t_1, t_2 \in [0,T]$ with $t_1 \leq t_2$ it holds that 
 \begin{align}
 \label{eq:P:int}
\P \Big( \smallint\nolimits_{t_1}^{t_2} f(s) \, dW_s = f(t_2) W_{t_2} - f(t_1) W_{t_1} - \smallint\nolimits_{t_1}^{t_2} f'(s) W_s \, ds \Big) =1
 \end{align}
and
 \begin{align}
\P \Big( Y_2  = - \smallint\nolimits_{a}^{b} f'(s) B_s \, ds \Big)=1.
 \end{align}
Item~\eqref{item:BM:indep} in Lemma~\ref{lemma:BM:gen} therefore proves that
\begin{align}
\label{eq:Y2:indep}
Y_2 \qquad \text{and} \qquad \tilde{W}
\end{align}
are independent on $(\Omega, \mathcal{F}, \P)$. The fact that $Y_2$ is a Gaussian random variable with mean $0$ hence implies that
\begin{align}
\label{eq:indep:gen}
\P_{(\tilde{W}, Y_2)} = \P_{\tilde{W}} \otimes \P_{Y_2}= \P_{\tilde{W}} \otimes \P_{-Y_2}= \P_{(\tilde{W}, -Y_2)}.
\end{align}
Next observe that \eqref{eq:lem:Y1} and \eqref{eq:P:int} assure that there exists a measurable function $\Phi_1 \colon C([0,a] \cup [b,T],\R) \to \R$ such that 
\begin{align}
\label{eq:Y1}
\P\big(Y_1 = \Phi_1(\tilde{W})\big)=1.
\end{align}
This, Lemma~\ref{lemma:P:gen} (with $\Omega = \Omega$, $S_1= C([0,a] \cup [b,T],\R)$, $S_2 = \R$, $X_1= \tilde{W}$, $X_2 = Y_2$, $X_2'= Y_2$, $X_2''= -Y_2$, $\varphi = u$, and $\Phi = (C([0,a] \cup [b,T], \R) \times \R \ni (w,y) \mapsto  z_T(\cos(\psi(\Phi_1(w)+y))) \in \R)$ in the notation of Lemma~\ref{lemma:P:gen}), \eqref{eq:zT}, and \eqref{eq:indep:gen} show that
\begin{align}
\label{eq:Xtwo:u}
\begin{split}
&\E \! \left[\big| \Xpsitwo_T - u((W_s)_{s \in [0,a] \cup [b,T]})\big|\right] = \E \! \left[\big| z_T(\cos(\psi(Y_1 +Y_2))) - u(\tilde{W})\big|\right]\\
& = \E \! \left[\big| z_T(\cos(\psi(\Phi_1(\tilde{W}) +Y_2))) - u(\tilde{W})\big|\right] \\
& \geq \tfrac{1}{2} \, \E \! \left[ \big| z_T(\cos(\psi(\Phi_1(\tilde{W}) +Y_2)))  - z_T(\cos(\psi(\Phi_1(\tilde{W}) -Y_2))) \big| \right] \\
& = \tfrac{1}{2} \, \E \Big[ \big| z_T(\cos(\psi(Y_1 +Y_2)))  - z_T(\cos(\psi(Y_1 -Y_2))) \big|\Big].
\end{split}
\end{align}
Corollary~\ref{cor:z:diff} therefore ensures that
\begin{align}
\label{eq:mean:gen}
\begin{split}
&\E \! \left[\big| \Xpsitwo_T - u((W_s)_{s \in [0,a] \cup [b,T]})\big|\right] \\
&\geq 2 (T-\tau_2) \,\E \big[\! \left| \cos(\psi(Y_1 +Y_2))  - \cos(\psi(Y_1 -Y_2))\right|\!\big] .
\end{split}
\end{align}
Moreover, note that \eqref{eq:Y1} and \eqref{eq:Y2:indep} demonstrate that $Y_1$ and $Y_2$ are independent on $(\Omega, \mathcal{F}, \P)$. The fact that $Y_1$ and $Y_2$ are centered Gaussian  distributed random variables hence shows that
\begin{align}
\begin{split}
& \E \big[ \! \left|\cos(\psi(Y_1 +Y_2))  - \cos(\psi(Y_1 -Y_2))\right| \!\big]  \\
&= \int_{\R} \int_{\R} \left|\cos(\psi(x+y)) - \cos(\psi(x-y)) \right|  \mathcal{N}_{0,\sigma_1}(dx) \, \mathcal{N}_{0,\sigma_2}(dy)\\
& \geq \int_{[0,1]} \int_{[c-1,c+1]} \left|\cos(\psi(x+y)) - \cos(\psi(x-y)) \right|  \mathcal{N}_{0,\sigma_1}(dx) \, \mathcal{N}_{0,\sigma_2}(dy)\\
& =  \int_{[0,1]} \int_{[c-1,c+1]} \left|\cos\!\left( \tfrac{x+y-c}{\sqrt{\beta}} \right) - \cos\!\left(\tfrac{x-y-c}{\sqrt{\beta}}\right)\!\right|  \mathcal{N}_{0,\sigma_1}(dx) \, \mathcal{N}_{0,\sigma_2}(dy)\\
& =  \int_{[0,1]} \int_{[c-1,c+1]} \left|\cos\!\left(\psi(x) + \tfrac{y}{\sqrt{\beta}} \right) - \cos\!\left(\psi(x)-\tfrac{y}{\sqrt{\beta}}\right) \! \right|  \mathcal{N}_{0,\sigma_1}(dx) \, \mathcal{N}_{0,\sigma_2}(dy).
\end{split}
\end{align}
The fact that $ \forall \, v, w \in \R \colon \cos(v) - \cos(w) = -2 \sin\!\big(\frac{v-w}{2}\big) \sin\!\big(\frac{v+w}{2}\big)$ therefore assures that
\begin{align}
\label{eq:normal:gen}
\begin{split}
& \E \big[ \! \left|\cos(\psi(Y_1 +Y_2))  - \cos(\psi(Y_1 -Y_2))\right| \!\big] \\
& \geq  2 \int_{[0,1]} \int_{[c-1,c+1]} \left|\sin(\psi(x))\right| \left|\sin\!\left(\tfrac{y}{\sqrt{\beta}}\right)\!\right|  \mathcal{N}_{0,\sigma_1}(dx) \, \mathcal{N}_{0,\sigma_2}(dy)\\
& \geq 2 \int_{[0,1]} \int_{A} \left|\sin(\psi(x))\right| \left|\sin\!\left(\tfrac{y}{\sqrt{\beta}}\right)\!\right| \mathcal{N}_{0,\sigma_1}(dx) \, \mathcal{N}_{0,\sigma_2}(dy)\\
& \geq  \mathcal{N}_{0,\sigma_1}(A) \int_{[0,1]}  \left|\sin\!\left(\tfrac{y}{\sqrt{\beta}}\right)\!\right| \mathcal{N}_{0,\sigma_2}(dy).
\end{split}
\end{align}
In addition, observe that Item~\eqref{item:BM:2} in  Lemma~\ref{lemma:BM:gen} proves that
\begin{align}
 \frac{ T^3 \beta }{12} = \frac{\varepsilon^3}{12}  \leq \frac{(b-a)^3}{12} \leq \sigma_2 \leq \frac{(b-a)^3}{3} = \frac{\varepsilon^3}{3} = \frac{T^3 \beta}{3}.
\end{align}
This implies that 
\begin{align}
\label{eq:sin:gen}
&\int_{[0,1]} \left|\sin\!\left(\tfrac{y}{\sqrt{\beta}}\right)\!\right| \mathcal{N}_{0,\sigma_2}(dy) = \int_{[0,1]}  \left|\sin\!\left(\tfrac{y}{\sqrt{\beta}}\right)\!\right| \!\frac{1}{\sqrt{2\sigma_2 \pi}} \, e^{-\frac{y^2}{2 \sigma_2}}\,  dy  \nonumber  \\
& \geq \int_{[0, \sqrt{\beta}]}  \left|\sin\!\left(\tfrac{y}{\sqrt{\beta}}\right)\!\right|\!  \frac{\sqrt{3}}{\sqrt{2 T^3 \beta \pi}} \, e^{-\frac{y^2}{2 \sigma_2}}\,  dy \geq \int_{[0, \sqrt{\beta}]}  \left|\sin\!\left(\tfrac{y}{\sqrt{\beta}}\right)\!\right|\!  \frac{\sqrt{3}}{\sqrt{2 T^3 \beta \pi}} \, e^{-\frac{6 y^2}{T^3 \beta}}\,  dy \nonumber \\
&= \frac{\sqrt{3}}{\sqrt{2  T^3 \pi}}  \int_{[0,1]}  \left|\sin(y)\right|    e^{-\frac{6y^2}{T^3 }}\,  dy.
\end{align}
Moreover, Item~\eqref{item:BM:1} in Lemma~\ref{lemma:BM:gen} shows that
\begin{align}
\tfrac{\alpha}{2} \leq \sigma_1 \leq \alpha.
\end{align}
Lemma~\ref{lem:sin}  hence proves that
\begin{align}
\begin{split}
\mathcal{N}_{0,\sigma_1}(A) &= \int_A \frac{1}{\sqrt{2 \sigma_1 \pi}} \, e^{-\frac{x^2}{2\sigma_1}} \, dx \geq \int_A \frac{1}{\sqrt{2 \sigma_1 \pi}} \, e^{-\frac{x^2}{\alpha}} \, dx\\
& \geq \int_A \frac{1}{\sqrt{2 \alpha \pi}} \, e^{-\frac{x^2}{\alpha}} \, dx \geq \int_{[c+\nicefrac{1}{2}, c+1]} \frac{1}{\sqrt{2 \alpha \pi}} \, e^{-\frac{x^2}{\alpha}} \, dx.
\end{split}
\end{align}
Combining this with \eqref{eq:mean:gen}, \eqref{eq:normal:gen}, and \eqref{eq:sin:gen} yields that
\begin{align}
\begin{split}
&\E \! \left[\big| \Xpsitwo_T - u((W_s)_{s \in [0,a] \cup [b,T]})\big|\right] \\
&\geq 2(T-\tau_2) \, \mathcal{N}_{0,\sigma_1}(A) \left[ \int_{[0,1]} \left|\sin\!\left(\tfrac{y}{\sqrt{\beta}}\right)\!\right| \mathcal{N}_{0,\sigma_2}(dy) \right] \\
&\geq  2(T-\tau_2) \Bigg[ \frac{1}{\sqrt{2 \alpha \pi}} \int_{[c+\nicefrac{1}{2}, c+1]} e^{-\frac{x^2}{\alpha}} \, dx \Bigg] \!\left[ \frac{\sqrt{3}}{\sqrt{2  T^3 \pi}}  \int_{[0,1]}  \left|\sin(y)\right|     e^{-\frac{6y^2}{T^3 }}\,  dy\right] \\
& = \left[ 2(T-\tau_2) \cdot \frac{1}{\sqrt{2 \alpha \pi}} \cdot \frac{\sqrt{3}}{\sqrt{2  T^3 \pi}}   \right]\! \bigg[  \int_{[c+\nicefrac{1}{2}, c+1]} e^{-\frac{x^2}{\alpha}} \, dx \bigg] \!\left[  \int_{[0,1]}   \left|\sin(y)\right|    e^{-\frac{6y^2}{T^3 }}\,  dy\right] \\
& = \left[ \frac{(T-\tau_2)}{\sqrt{ \alpha \pi}} \cdot \frac{\sqrt{3}}{\sqrt{ T^3 \pi}}   \right]\! \bigg[  \int_{[c+\nicefrac{1}{2}, c+1]} e^{-\frac{x^2}{\alpha}} \, dx \bigg] \!\left[  \int_{[0,1]}   \left|\sin(y)\right|     e^{-\frac{6y^2}{T^3 }}\,  dy\right] \\
& = \frac{\sqrt{3}(T-\tau_2)}{\pi\sqrt{T^3 \alpha}}  \bigg[  \int_{[c+\nicefrac{1}{2}, c+1]} e^{-\frac{x^2}{\alpha}} \, dx \bigg] \!\left[  \int_{[0,1]}  \left|\sin(y)\right|     e^{-\frac{6y^2}{T^3 }}\,  dy\right]  > 0.
\end{split}
\end{align}
The proof of Lemma~\ref{lemma:lower:1:gen} is thus completed.
\end{proof}

\begin{lemma}
\label{lemma:lower:2:gen}
Assume the setting in Section~\ref{setting:gen} and let $a\in [0, \tau)$, $b \in (a,  \tau]$, $ \varepsilon \in (0, b-a]$, $c \in [2, \infty)$, $\psi \in  C^{\infty}(\R,\R)$ satisfy for all $x \in [c-2, c+2]$ that $\psi(x) = \frac{T^{3/2}}{\varepsilon^{3/2}} \cdot (x-c)$. Then 
\begin{align}
\begin{split}
&\inf_{\substack{u \colon C([0,a] \cup [b,T], \R) \to \R \\  \text{ measurable}}} \E \! \left[\big| \Xpsitwo_T - u((W_s)_{s \in [0,a] \cup [b,T]})\big|\right] \\
&\geq \frac{\sqrt{3}(T-\tau_2)}{\pi\sqrt{T^3 \alpha}} \! \left[  \int_{c+\nicefrac{1}{2}}^{c+1} e^{-\frac{x^2}{\alpha}} \, dx \right] \!\left[  \int_{0}^1  \left|\sin(y)\right|   e^{-\frac{6y^2}{T^3 }}\,  dy\right]
> 0.
\end{split}
\end{align}
\end{lemma}
\begin{proof}[Proof of Lemma~\ref{lemma:lower:2:gen}]
Throughout this proof let $a_1 \in [a, b)$, $b_1 \in (a_1,b]$ be real numbers which satisfy that $(b_1-a_1) = \varepsilon$. Note that Lemma~\ref{lemma:lower:1:gen} proves that
\begin{align}
\begin{split}
&\inf_{\substack{u \colon C([0,a] \cup [b,T], \R) \to \R \\  \text{ measurable}}} \E \! \left[\big| \Xpsitwo_T - u((W_s)_{s \in [0,a] \cup [b,T]})\big|\right] \\
&\geq \inf_{\substack{u \colon C([0,a_1] \cup [b_1,T], \R) \to \R \\  \text{ measurable}}} \E \! \left[\big| \Xpsitwo_T - u((W_s)_{s \in [0,a_1] \cup [b_1,T]})\big|\right] \\
&\geq  \frac{\sqrt{3}(T-\tau_2)}{\pi\sqrt{T^3 \alpha}} \! \left[  \int_{c+\nicefrac{1}{2}}^{c+1} e^{-\frac{x^2}{\alpha}} \, dx \right] \!\left[  \int_{0}^1  \left|\sin(y)\right|   e^{-\frac{6y^2}{T^3 }}\,  dy\right]  > 0.
\end{split}
\end{align}
The proof of Lemma~\ref{lemma:lower:2:gen} is thus completed.
\end{proof}

The next result, Corollary~\ref{cor:lower:1:gen}, follows directly from Lemma~\ref{lemma:lower:2:gen}.

\begin{cor}
\label{cor:lower:1:gen}
Assume the setting in Section~\ref{setting:gen} and let  $ (\varepsilon_n)_{n \in \N} \subseteq (0, \tau]$, $\psi \in  C^{\infty}(\R,\R)$ satisfy for all $n \in \N$,  $x \in [5n-2, 5n+2]$ that  $\psi(x) = \frac{T^{3/2}}{|\varepsilon_n|^{3/2}} \cdot (x-5n)$. Then it holds for all $n \in \N$  that
\begin{align}
\begin{split}
&\inf_{\substack{a,b \in [0, \tau], \\
	b-a \geq \varepsilon_n }} \inf_{\substack{u \colon C([0,a] \cup [b,T], \R) \to \R \\  \text{ measurable}}}\E \! \left[\big| \Xpsitwo_T - u((W_s)_{s \in [0,a] \cup [b,T]})\big|\right] \\
&\geq \frac{\sqrt{3}(T-\tau_2)}{\pi\sqrt{T^3 \alpha}} \! \left[  \int_{5n+\nicefrac{1}{2}}^{5n+1} e^{-\frac{x^2}{\alpha}} \, dx \right] \!\left[  \int_{0}^1  \left|\sin(y)\right|   e^{-\frac{6y^2}{T^3 }}\,  dy\right] > 0.
\end{split}
\end{align}
\end{cor}

\subsection{Asymptotic lower bounds for strong approximation errors for two-dimensional SDEs}

\begin{lemma}
\label{thm:lower:seq1}
Assume the setting in Section~\ref{setting:gen} and let $(\varepsilon_n)_{n \in \N} \subseteq (0,  \tau]$ and $(\delta_n)_{n \in \N} \subseteq \R$ be non-increasing sequences with $\limsup_{n \to \infty} \delta_n \leq  0$.  Then there exist a function $\psi \in C^{\infty}(\R,\R)$ and a natural number $n_0 \in \N$ such that for all $n \in   \N$ it holds that
\begin{align}
\inf_{\substack{a,b \in [0, \tau], \\
b-a \geq \varepsilon_n }} \inf_{\substack{u \colon C([0,a] \cup [b,T], \R) \to \R \\  \text{ measurable}}}\E \! \left[\big| \Xpsitwo_T - u((W_s)_{s \in [0,a] \cup [b,T]})\big|\right] > \mathbbm{1}_{[n_0, \infty)}(n)  \max\{\delta_n,0\}.
\end{align}
\end{lemma}
\begin{proof}[Proof of Lemma~\ref{thm:lower:seq1}]
Note that the assumption that $\limsup_{n \to \infty} \delta_n \leq  0$ ensures that $\limsup_{n \to \infty} \max\{\delta_n,0\} =  0$. This shows that there exists a strictly increasing function $n \colon \N \to \N$ which satisfies for all $m \in \N$ that 
\begin{align}
\label{eq:n(m)}
\frac{\sqrt{3}(T-\tau_2)}{\pi\sqrt{T^3 \alpha}} \! \left[  \int_{5m+\nicefrac{1}{2}}^{5m+1} e^{-\frac{x^2}{\alpha}} \, dx \right] \!\left[  \int_{0}^1  \left|\sin(y)\right|  e^{-\frac{6y^2}{T^3 }}\,  dy\right] > \max\{\delta_{n(m)},0\}.
\end{align}
Next let $\psi \in C^{\infty}(\R, \R)$ be a function which satisfies for all $m \in \N$,  $x \in [5m-2, 5m+2]$ that 
\begin{align}
\label{eq:psi:nm}
\psi(x) = \tfrac{T^{3/2}}{|\varepsilon_{n(m+1)}|^{3/2}} \cdot (x-5m).
\end{align}
Observe that Corollary~\ref{cor:lower:1:gen} (with $\varepsilon_m = \varepsilon_{n(m+1)}$ for $m \in \N$ in the notation of Corollary~\ref{cor:lower:1:gen}), \eqref{eq:psi:nm}, and \eqref{eq:n(m)} prove that for all $m \in \N$, $k \in [n(m),n(m+1)] \cap \N$ it holds that 
\begin{align}
\begin{split}
&\inf_{\substack{a,b \in [0, \tau], \\
b-a \geq \varepsilon_k }} \inf_{\substack{u \colon C([0,a] \cup [b,T], \R) \to \R \\  \text{ measurable}}}\E \! \left[\big| \Xpsitwo_T - u((W_s)_{s \in [0,a] \cup [b,T]})\big|\right] \\
& \geq \inf_{\substack{a,b \in [0, \tau], \\
b-a \geq \varepsilon_{n(m+1)} }} \inf_{\substack{u \colon C([0,a] \cup [b,2], \R) \to \R \\  \text{ measurable}}}\E \! \left[\big| \Xpsitwo_T - u((W_s)_{s \in [0,a] \cup [b,T]})\big|\right] \\ &                                                                                                                                                                                                                                                                                                                                                                                                                                                                                                                                                                                                                                                                                                                                                                                                                                                                                                                                                                                                                                                                                                                                                                                                                                                                                                                                                                                                                                                                                                                                                                                                                                                                                                                                                                                                                                                                                                                                                                                                                                                                                                                                                                                                                                                                                                                                                                                                                                                                                                                                                                                                                                                                                                                                                                                                                                                                                                                                                                                                                                                                                                                                                                                                                                                                                                                                                                                                                                                                                                                                                                                                                                                                                                                                                                                                                                                                                                                                                                                                                                                                                                                                                                                                                                                                                                                                                                                                                                                                                                                                                                                                                                                                                                                                                                                                                                                                                                                                                                                                                                                                                                                                                                                                                                                                                                                                                                                                                                                                                                                                                                                                                                                                                                                                                                                                                                                                                                                                                                                                                                                                                                                                                                                                                                                                                                                                                                                                                                                                                                                                                                                                                                                                                                                                                                                                                                                                                                                                                                                                                                                                                                                                                                                                                                                                                                                                                                                                                                                                                                                                                                                                                                                                                                                                                                                                                                                                                                                                                                                                                                                                                                                                                                                                                                                                                                                                                                                                                                                                                                                                                                                                                                                                                                                                                                                                                                                                                                                                                                                                                                                                                                                                                                                                                                                                                                                                                                                                                                                                                                                                                                                                                                                                                                                                                                                                                                                                                                                                                                                                                                                                                                                                                                                                                                                                                                                                                                                                                                                                                                       \geq \frac{\sqrt{3}(T-\tau_2)}{\pi\sqrt{T^3 \alpha}} \! \left[  \int_{5m+\nicefrac{1}{2}}^{5m+1} e^{-\frac{x^2}{\alpha}} \, dx \right] \!\left[  \int_{0}^1  \left|\sin(y)\right|   e^{-\frac{6y^2}{T^3 }}\,  dy\right]   > \max\{\delta_{n(m)},0\} \\
&\geq \max\{\delta_k,0\}.
\end{split}
\end{align}
This implies that for all $k \in [n(1),\infty) \cap \N$ it holds that 
\begin{align}
\label{eq:n(1)}
\inf_{\substack{a,b \in [0, \tau], \\
b-a \geq \varepsilon_k }} \inf_{\substack{u \colon C([0,a] \cup [b,T], \R) \to \R \\  \text{ measurable}}}\E \! \left[\big| \Xpsitwo_T - u((W_s)_{s \in [0,a] \cup [b,T]})\big|\right] > \max\{\delta_k,0\}.
\end{align}
The assumption that $(\varepsilon_n)_{n \in \N} \subseteq (0,  \infty)$ is non-increasing hence proves that for all $k \in [1,n(1)] \cap \N$ it holds that
\begin{align}
\label{eq:n(1):2}
\begin{split}
&\inf_{\substack{a,b \in [0, \tau], \\
b-a \geq \varepsilon_k }} \inf_{\substack{u \colon C([0,a] \cup [b,T], \R) \to \R \\  \text{ measurable}}}\E \! \left[\big| \Xpsitwo_T - u((W_s)_{s \in [0,a] \cup [b,T]})\big|\right] \\
&\geq \inf_{\substack{a,b \in [0, \tau], \\
b-a \geq \varepsilon_{n(1) }}} \inf_{\substack{u \colon C([0,a] \cup [b,T], \R) \to \R \\  \text{ measurable}}}\E \! \left[\big| \Xpsitwo_T - u((W_s)_{s \in [0,a] \cup [b,T]})\big|\right]  \\
&> \max\{\delta_{n(1)},0\} \geq 0.
\end{split}
\end{align}
Combining \eqref{eq:n(1)} and \eqref{eq:n(1):2}  completes the proof of Lemma~\ref{thm:lower:seq1}.
\end{proof}

In the next result, Lemma~\ref{lem:lower:seq1} below, we generalize the result of  Lemma~\ref{thm:lower:seq1} by removing the restriction  that the sequence $(\varepsilon_n)_{n \in \N} \subseteq (0,  \tau]$ appearing in Lemma~\ref{thm:lower:seq1} has to be non-increasing.

\begin{lemma}
\label{lem:lower:seq1}
Assume the setting in Section~\ref{setting:gen}, let $(\varepsilon_n)_{n \in \N} \subseteq (0,  \tau]$ be a sequence, and let $(\delta_n)_{n \in \N} \subseteq \R$ be a non-increasing sequence with $\limsup_{n \to \infty} \delta_n \leq  0$.  Then there exist a function $\psi \in C^{\infty}(\R,\R)$ and a natural number $n_0 \in \N$ such that for all $n \in  \N$ it holds that
\begin{align}
\inf_{\substack{a,b \in [0, \tau], \\
b-a \geq \varepsilon_n }} \inf_{\substack{u \colon C([0,a] \cup [b,T], \R) \to \R \\  \text{ measurable}}}\E \! \left[\big| \Xpsitwo_T - u((W_s)_{s \in [0,a] \cup [b,T]})\big|\right] > \mathbbm{1}_{[n_0, \infty)}(n) \max\{\delta_n,0\}.
\end{align}
\end{lemma}
\begin{proof}[Proof of Lemma~\ref{lem:lower:seq1}]
Throughout this proof let $(\tilde{\varepsilon}_n)_{n \in \N} \subseteq (0, \tau]$ be the sequence which satisfies for all $n \in \N$ that 
\begin{align}
\tilde{\varepsilon}_{n} = \min\{\varepsilon_{1}, \varepsilon_2, \ldots, \varepsilon_n\}.
\end{align}
This ensures that $(\tilde{\varepsilon}_n)_{n \in \N} \subseteq (0, \tau]$ is a non-increasing sequence. Lemma~\ref{thm:lower:seq1} (with $\varepsilon_n = \tilde{\varepsilon}_n$ and $\delta_n = \delta_n$ for $n \in \N$ in the notation of Lemma~\ref{thm:lower:seq1}) hence  proves that there exist a function $\psi \in C^{\infty}(\R,\R)$ and a natural number $n_0 \in \N$ such that for all $n \in \N$ it holds that
\begin{align}
\begin{split}
&\inf_{\substack{a,b \in [0, \tau], \\
b-a \geq \varepsilon_n }} \inf_{\substack{u \colon C([0,a] \cup [b,T], \R) \to \R \\  \text{ measurable}}}\E \! \left[\big| \Xpsitwo_T - u((W_s)_{s \in [0,a] \cup [b,T]})\big|\right] \\
&\geq \inf_{\substack{a,b \in [0, \tau], \\
b-a \geq \tilde{\varepsilon}_n }} \inf_{\substack{u \colon C([0,a] \cup [b,T], \R) \to \R \\  \text{ measurable}}}\E \! \left[\big| \Xpsitwo_T - u((W_s)_{s \in [0,a] \cup [b,T]})\big|\right] \\
&> \mathbbm{1}_{[n_0, \infty)}(n) \max\{\delta_n,0\}.
\end{split}
\end{align}
The proof of Lemma~\ref{lem:lower:seq1} is thus completed.
\end{proof}

The next result, Corollary~\ref{cor:lower:seq2} below,  generalizes the result of  Lemma~\ref{lem:lower:seq1} by eliminating the condition  that the sequence $(\delta_n)_{n \in \N} \subseteq \R$ appearing in Lemma~\ref{lem:lower:seq1}  has to be non-increasing.

\begin{cor}
\label{cor:lower:seq2}
Assume the setting in Section~\ref{setting:gen} and let $(\varepsilon_n)_{n \in \N} \subseteq (0,  \tau]$ and $(\delta_n)_{n \in \N} \subseteq \R$ be  sequences with $\limsup_{n \to \infty} \delta_n \leq  0$.  Then there exist a function $\psi \in C^{\infty}(\R,\R)$ and a natural number $n_0 \in \N$ such that for all $n \in  \N$ it holds that
\begin{align}
\inf_{\substack{a,b \in [0, \tau], \\
b-a \geq \varepsilon_n }} \inf_{\substack{u \colon C([0,a] \cup [b,T], \R) \to \R \\  \text{ measurable}}}\E \! \left[\big| \Xpsitwo_T - u((W_s)_{s \in [0,a] \cup [b,T]})\big|\right] > \mathbbm{1}_{[n_0, \infty)}(n) \max\{\delta_n,0\}.
\end{align}
\end{cor}
\begin{proof}[Proof of Corollary~\ref{cor:lower:seq2}]
Throughout this proof let $(\tilde{\delta}_n)_{n \in \N} \subseteq (-\infty, \infty]$ be the sequence of extended real numbers which satisfies for all $n \in \N$ that
\begin{align}
\label{eq:delta}
\tilde{\delta}_n = \sup\{ \delta_n, \delta_{n+1}, \delta_{n+2}, \ldots \}.
\end{align}
The assumption that $\limsup_{n \to \infty} \delta_n \leq  0$ hence ensures that $\forall \, n \in \N \colon \tilde{\delta}_n \in \R$, that
\begin{align}
\limsup_{n \to \infty} \tilde{\delta}_n = \lim_{n \to \infty} \tilde{\delta}_n = \limsup_{n \to \infty} \delta_n \leq  0,
\end{align}
and that $(\tilde{\delta}_n)_{n \in \N}$ is a non-increasing sequence. This allows us to apply Lemma~\ref{lem:lower:seq1} (with $\varepsilon_n = \varepsilon_n$ and $\delta_n = \tilde{\delta}_n$ for $n \in \N$ in the notation of Lemma~\ref{lem:lower:seq1}) to obtain that there exist a function $\psi \in C^{\infty}(\R,\R)$ and a natural number $n_0 \in \N$ such that for all $n \in  \N$ it holds that
\begin{align}
\begin{split}
&\inf_{\substack{a,b \in [0, \tau], \\
b-a \geq \varepsilon_n }} \inf_{\substack{u \colon C([0,a] \cup [b,T], \R) \to \R \\  \text{ measurable}}}\E \! \left[\big| \Xpsitwo_T - u((W_s)_{s \in [0,a] \cup [b,T]})\big|\right]  \\
&> \mathbbm{1}_{[n_0, \infty)}(n) \max\{\tilde{\delta}_n,0\} \geq \mathbbm{1}_{[n_0, \infty)}(n) \max\{\delta_n,0\}.
\end{split}
\end{align}
The proof of Corollary~\ref{cor:lower:seq2} is thus completed.
\end{proof}

\subsection{Non-asymptotic lower bounds for strong approximation errors for two-dimensional SDEs}
\label{sub:non:asymp}

\begin{lemma}
\label{lem:measurable}
Assume the setting in Section~\ref{setting:gen} and let $\psi \in C^{\infty}(\R,\R)$. Then there exists a measurable function $\Phi \colon C([0,T],\R) \to \R$ such that
\begin{align}
\P\Big(\Xpsitwo_T = \Phi\big((W_s)_{s \in [0,T]}\big)\Big)=1.
\end{align}
\end{lemma}
\begin{proof}[Proof of Lemma~\ref{lem:measurable}]
Note that Lemma~\ref{lemma:xtwo:z:gen} proves that there exists a measurable function $\phi \colon \R \to  \R$ such that 
\begin{align}
\label{eq:F}
\P\Big(\Xpsitwo_T = \phi\big(\Xpsione_{\tau_1}\big)\Big)=1.
\end{align}
Moreover, Item~\eqref{item:tau1:X} in Lemma~\ref{lemma:X:iden:gen}   and  Item~\eqref{item:Ito} in Lemma~\ref{lemma:BM:gen} ensure that 
it holds $\P$-a.s.~that 
\begin{align}
\begin{split}
\Xpsione_{\tau_1} &=  \int_{0}^{\tau_1} f(s) \, dW_s = f(\tau_1) W_{\tau_1} - f(0) W_{0} - \int_{0}^{\tau_1} f'(s) W_s \, ds\\
& = f(\tau_1) W_{\tau_1} - \int_{0}^{\tau_1} f'(s) W_s \, ds.
\end{split}
\end{align}
Combining this with \eqref{eq:F} completes the proof of Lemma~\ref{lem:measurable}.
\end{proof}

\begin{cor}
\label{cor:gen:C}
Assume the setting in Section~\ref{setting:gen} and  let $(\varepsilon_n)_{n \in \N} \subseteq (0,  \tau]$ and $(\delta_n)_{n \in \N} \subseteq \R$ be sequences with $\limsup_{n \to \infty} \delta_n \leq  0$.  Then there exist a function $\psi \in C^{\infty}(\R,\R)$, a real number $c \in (0, \infty)$, a measurable function $\Phi \colon C([0,T],\R) \to \R$,   and a continuous $\mathbb{F}$-adapted  stochastic process $Z \colon [0,T] \times \Omega \to \R$ such that for all $n \in \N$, $t \in [0,T]$ it holds that
\begin{align}
\P\Big(Z_T = \Phi \big((W_s)_{s \in [0,T]}\big)\Big)=1,
\end{align}
\begin{align}
 \P \! \left( \Xpsione_t = \textstyle \int_0^t  f(\tfrac{Z_s}{c}) \, dW_s \displaystyle \right)=1,
\end{align} 
\begin{align}
 \P \! \left( Z_t = \textstyle\int_0^t c+   c\, g\big(\tfrac{Z_s}{c}\big) [\cos(\psi(\Xpsione_s)) +1] \, ds \displaystyle \right)=1,
\end{align}
and 
\begin{align}
\inf_{\substack{a,b \in [0, \tau], \\
b-a \geq \varepsilon_n }} \inf_{\substack{u \colon C([0,a] \cup [b,T], \R) \to \R \\  \text{ measurable}}}\E \Big[\!\left| Z_T - u((W_s)_{s \in [0,a] \cup [b,T]})\right|\!\Big] \geq \delta_n.
\end{align}
\end{cor}
\begin{proof}[Proof of Corollary~\ref{cor:gen:C}]
First, note that Corollary~\ref{cor:lower:seq2} proves that there exist a function $\psi \in C^{\infty}(\R,\R)$ and a natural number $n_0 \in \N$ such that for all $n \in  \N$ it holds that
\begin{align}
\label{eq:n0}
\inf_{\substack{a,b \in [0, \tau], \\
b-a \geq \varepsilon_n }} \inf_{\substack{u \colon C([0,a] \cup [b,T], \R) \to \R \\  \text{ measurable}}}\E \! \left[\big| \Xpsitwo_T - u((W_s)_{s \in [0,a] \cup [b,T]})\big|\right] > \mathbbm{1}_{[n_0, \infty)}(n)  \max\{\delta_n,0\}.
\end{align}
Next let $(e_n)_{n \in \N} \subseteq (0,\infty)$ be the sequence which satisfies for all $n \in \N$  that 
\begin{align}
e_n = \inf_{\substack{a,b \in [0, \tau], \\
b-a \geq \varepsilon_n }} \inf_{\substack{u \colon C([0,a] \cup [b,T], \R) \to \R \\  \text{ measurable}}}\E \! \left[\big| \Xpsitwo_T - u((W_s)_{s \in [0,a] \cup [b,T]})\big|\right],
\end{align}
 let $c\in (0, \infty)$ be the real number given by 
 \begin{align}
 c= \max \!\left( \left\{ 1, \tfrac{\max\{\delta_1,0\}}{e_1}, \tfrac{\max\{\delta_2,0\}}{e_2}, \ldots, \tfrac{\max\{\delta_{n_0},0\}}{e_{n_0}} \right\}\right),
 \end{align}
and let $Z \colon [0,T] \times \Omega \to \R$ be the stochastic process which satisfies  for all $t \in [0,T]$ that $Z_t = c \Xpsitwo_t$. Note that for all $t \in [0,T]$ it holds that
\begin{align}
\label{eq:Xone:Z}
 \P \! \left( \Xpsione_t = \textstyle \int_0^t  f(\tfrac{Z_s}{c}) \, dW_s \displaystyle \right)=1
\end{align}
and 
\begin{align}
\label{eq:Xtwo:Z}
 \P \! \left( Z_t = \textstyle\int_0^t c+   c\, g\big(\tfrac{Z_s}{c}\big) [\cos(\psi(\Xpsione_s)) +1] \, ds \displaystyle \right)=1.
\end{align}
Next observe that Lemma~\ref{lem:measurable} and the fact that $Z_T = c \Xpsitwo_T$ prove that there exists a measurable function $\Phi \colon C([0,T],\R) \to \R$ such that
\begin{align}
\label{eq:Z:meas}
\P\Big(Z_T = \Phi \big((W_s)_{s \in [0,T]}\big)\Big)=1.
\end{align}
Moreover, note that \eqref{eq:n0} ensures that for all $n \in  \{n_0, n_0+1, \ldots\}$ it holds that
\begin{align}
\label{eq:c:en:1}
\begin{split}
c \cdot e_{n} & = \max \!\left( \left\{ 1, \tfrac{\max\{\delta_1,0\}}{e_1}, \tfrac{\max\{\delta_2,0\}}{e_2}, \ldots, \tfrac{\max\{\delta_{n_0},0\}}{e_{n_0}} \right\}\right) \cdot e_n\\& \geq e_n > \mathbbm{1}_{[n_0, \infty)}(n)  \max\{\delta_n,0\} =   \max\{\delta_n,0\}.
\end{split}
\end{align}
In addition, observe that for all $n \in \{1, 2, \ldots, n_0\}$ it holds that
\begin{align}
\label{eq:c:en:2}
\begin{split}
c \cdot e_{n} & = \max \!\left( \left\{ 1, \tfrac{\max\{\delta_1,0\}}{e_1}, \tfrac{\max\{\delta_2,0\}}{e_2}, \ldots, \tfrac{\max\{\delta_{n_0},0\}}{e_{n_0}} \right\}\right) \cdot e_n\\& \geq \tfrac{\max\{\delta_n,0\}}{e_n} \cdot e_n  =   \max\{\delta_n,0\}.
\end{split}
\end{align}
Combining \eqref{eq:c:en:1} and \eqref{eq:c:en:2} shows that for all $n \in \N$ it holds that $ c\cdot e_n \geq \max\{\delta_n,0\} \geq \delta_n$. Hence, we obtain that for all $n \in \N$ it holds that 
\begin{align}
\begin{split}
&\inf_{\substack{a,b \in [0, \tau], \\
b-a \geq \varepsilon_n }} \inf_{\substack{u \colon C([0,a] \cup [b,T], \R) \to \R \\  \text{ measurable}}}\E \Big[\!\left| Z_T - u((W_s)_{s \in [0,a] \cup [b,T]})\right|\!\Big] \\
&= \inf_{\substack{a,b \in [0, \tau], \\
b-a \geq \varepsilon_n }} \inf_{\substack{u \colon C([0,a] \cup [b,T], \R) \to \R \\  \text{ measurable}}}\E \! \left[\big| c \Xpsitwo_T - u((W_s)_{s \in [0,a] \cup [b,T]})\big|\right] \\
&= c \left(\inf_{\substack{a,b \in [0, \tau], \\
b-a \geq \varepsilon_n }} \inf_{\substack{u \colon C([0,a] \cup [b,T], \R) \to \R \\  \text{ measurable}}}\E \! \left[\big|  \Xpsitwo_T - \tfrac{1}{c} \cdot u((W_s)_{s \in [0,a] \cup [b,T]})\big|\right]  \right)\\
& = c \cdot e_n\geq \delta_n .
\end{split}
\end{align}
This and \eqref{eq:Xone:Z}--\eqref{eq:Z:meas} complete the proof of Corollary~\ref{cor:gen:C}.
\end{proof}

The next result, Lemma~\ref{lem:main}, follows  from Corollary~\ref{cor:gen:C} and from Corollary~\ref{cor:fg}.

\begin{lemma}
\label{lem:main}
Let $T \in (0, \infty)$, $\tau \in (0,T)$, $d \in \{2,3, \ldots\}$, $\xi \in \R^d$, $(\varepsilon_n)_{n \in \N} \subseteq (0, \tau]$, $(\delta_n)_{n \in \N} \subseteq \R$ satisfy $\limsup_{n \to \infty} \delta_n \leq 0$. Then there exist infinitely often differentiable and globally bounded functions $\mu, \sigma \colon \R^d \to \R^d$  and a measurable function  $\Phi \colon C([0,T],\R) \to \R$ such that for every probability space $(\Omega, \mathcal{F}, \P)$, every normal filtration $\mathbb{F}=(\mathbb{F}_t)_{t \in [0,T]}$ on $(\Omega, \mathcal{F}, \P)$, every standard $(\Omega, \mathcal{F}, \P, \mathbb{F})$-Brownian motion $W \colon  [0,T] \times \Omega \to \R$, every continuous $\mathbb{F}$-adapted stochastic process $X=(X^{(1)}, \ldots, X^{(d)}) \colon [0,T] \times \Omega \to \R^d$ with $\forall \, t \in [0,T] \colon \P(X_t = \xi + \int_0^t \mu(X_s) \, ds + \int_0^t \sigma(X_s) \, dW_s)=1$, and every $n \in \N$ it holds that 
\begin{align}
	\label{eq:meas}
\P\Big(X_T^{(1)}=  \Phi \big((W_s)_{s \in [0,T]}\big) \Big)=1
\end{align}  
and
\begin{align}
\label{eq:lem:main}
\begin{split}
\inf_{\substack{a,b \in [0, \tau], \\
b-a \geq \varepsilon_n }}   \inf_{\substack{u \colon C([0,a] \cup [b,T], \R)  \to \R \\  \text{ measurable}}}\E \! \left[\big| X_T^{(1)} - u((W_s)_{s \in [0,a] \cup [b,T]})\big|\right] \geq \delta_n.
\end{split}
\end{align}
\end{lemma}
\begin{proof}[Proof of Lemma~\ref{lem:main}]
Throughout this proof for all measurable spaces $(A, \mathcal{A})$ and $(B, \mathcal{B})$ let $\mathcal{M}(\mathcal{A}, \mathcal{B})$ be the set of all $\mathcal{A} \slash \mathcal{B}$-measurable functions from $A$ to $B$, 
let $f, g, \psi \in C^{\infty}(\R,\R)$, $c \in (0,\infty)$,  $\phi \in \mathcal{M}(\mathcal{B}(C([0,T],\R)) , \mathcal{B}(\R))$ satisfy  that for every probability space $(\Omega, \mathcal{F}, \P)$, every normal filtration $\mathbb{F}=(\mathbb{F}_t)_{t \in [0,T]}$ on $(\Omega, \mathcal{F}, \P)$, every standard $(\Omega, \mathcal{F}, \P, \mathbb{F})$-Brownian motion $W \colon  [0,T] \times \Omega \to \R$, every continuous $\mathbb{F}$-adapted stochastic process $X=(X^{(1)},  X^{(2)}) \colon [0,T] \times \Omega \to \R^2$ with $ \forall \, t \in [0,T] \colon \P  \big( X_t^{(1)} = \int_0^t c+   c\, g(\frac{X_s^{(1)}}{c}) [\cos(\psi(X_s^{(2)})) +1] \, ds \big)= \P \big(X_{t}^{(2)} = \int_0^t f(\frac{X_s^{(1)}}{c}) \, dW_s \big)=1$, and every $n \in \N$ it holds that 
\begin{align}
\label{eq:phi}
\P\Big(X_T^{(1)}=  \phi \big((W_s)_{s \in [0,T]}\big) \Big)=1,
\end{align}  
\begin{align}
\label{eq:inf}
\begin{split}
\inf_{\substack{a,b \in [0, \tau], \\
		b-a \geq \varepsilon_n }}   \inf_{\substack{u \colon C([0,a] \cup [b,T], \R)  \to \R \\  \text{ measurable}}}\E \! \left[\big| X_T^{(1)} - u((W_s)_{s \in [0,a] \cup [b,T]})\big|\right] \geq \delta_n,
\end{split}
\end{align}	
and $\sup_{x \in \R} (|f(x)|+ |g(x)|) < \infty$ (Corollary~\ref{cor:fg} and Corollary~\ref{cor:gen:C} assure that $f,g,\psi  \in C^{\infty}(\R,\R)$,  $c \in (0,\infty)$, $\phi \in \mathcal{M}(\mathcal{B}(C([0,T],\R)) , \mathcal{B}(\R))$ do indeed exist), let $P \colon \R^d \to \R$ be the function which satisfies for all $x = (x_1, \ldots, x_d) \in \R^d$ that $P(x)= x_1$, let $\Xi \in \R$ be the real number given by $\Xi = P(\xi)$,   
let $a, b, \mu, \sigma \colon \R^d \to \R^d$ be the functions which satisfy for all $x = (x_1, \ldots, x_d)\in \R^d$ that
\begin{align}
\label{eq:M}
a(x)=(c + c \, g(\tfrac{x_1}{c})[\cos(\psi(x_2))+1], 0, \ldots, 0),
\end{align}
\begin{align}
\label{eq:Sigma}
b(x)=(0,f(\tfrac{x_1}{c}), 0, \ldots, 0),
\end{align}
\begin{align}
\label{eq:mu}
\mu(x)=a(x-\xi), \qquad \text{and} \qquad \sigma(x)=b(x-\xi),
\end{align}
let $\Phi \colon C([0,T], \R) \to \R$ be the measurable function which satisfies for all $v \in C([0,T],\R)$ that 
\begin{align}
\label{eq:Phi}
\Phi(v) = \phi(v) + \Xi,
\end{align}
let $(\Omega, \mathcal{F}, \P)$ be a probability space, let  $\mathbb{F}=(\mathbb{F}_t)_{t \in [0,T]}$ be a normal filtration on $(\Omega, \mathcal{F}, \P)$,   let $W  \colon  [0,T] \times \Omega \to \R$ be a standard $(\Omega, \mathcal{F}, \P, \mathbb{F})$-Brownian motion, 
let $X=(X^{(1)}, \ldots, X^{(d)}) \colon [0,T] \times \Omega \to \R^d$ be a  continuous $\mathbb{F}$-adapted stochastic process which satisfies for all $t \in [0,T]$ that 
\begin{align}
\label{eq:X}
\P\Big(X_t = \xi + \smallint\nolimits_0^t \mu(X_s) \, ds + \smallint\nolimits_0^t \sigma(X_s) \, dW_s\Big)=1,
\end{align}
let  $n \in \N$, $a,b \in [0,\tau]$ be real numbers with $b-a \geq \varepsilon_n$, and let $Y =(Y^{(1)}, \ldots, Y^{(d)}) \colon [0,T] \times \Omega \to \R^d$ be the stochastic process which satisfies for all $t \in [0,T]$ that
\begin{align}
\label{eq:Y:xi}
Y_t = X_t - \xi.
\end{align}
Observe that \eqref{eq:X}, \eqref{eq:Y:xi}, \eqref{eq:mu}, and \eqref{eq:mu} ensure  that $Y  \colon [0,T] \times \Omega \to \R^d$ is a  continuous $\mathbb{F}$-adapted stochastic process    which satisfies for all $t \in [0,T]$ that
\begin{align}
\label{eq:Y}
\P\Big(Y_t =  \smallint\nolimits_0^t a(Y_s) \, ds + \smallint\nolimits_0^t b(Y_s) \, dW_s\Big)=1.
\end{align}
This, \eqref{eq:M}, and \eqref{eq:Sigma} show that for all $t \in [0,T]$ it holds that
\begin{align}
\P  \Big( Y_t^{(1)} = \smallint\nolimits_0^t c+   c\, g\big(\tfrac{Y_s^{(1)}}{c}\big) [\cos(\psi(Y_s^{(2)})) +1] \, ds \Big)=1
\end{align}
and 
\begin{align}
\P \Big(Y_{t}^{(2)} = \smallint\nolimits_0^t f\big(\tfrac{Y_s^{(1)}}{c}\big) \, dW_s \Big)=1.
\end{align}
Combining this with \eqref{eq:phi} and \eqref{eq:inf} demonstrates that 
\begin{align}
\label{eq:Y:meas}
\P\Big(Y_T^{(1)}=  \phi \big((W_s)_{s \in [0,T]}\big) \Big)=1
\end{align} 
and 
\begin{align}
\label{eq:inf:Y}
\begin{split}
\inf_{\substack{u \colon C([0,a] \cup [b,T], \R)  \to \R \\  \text{ measurable}}} \E \! \left[\big| Y_T^{(1)} - u((W_s)_{s \in [0,a] \cup [b,T]})\big|\right] \geq \delta_n.
\end{split}
\end{align}
In addition, observe that \eqref{eq:Y:xi}, \eqref{eq:Phi}, and  \eqref{eq:Y:meas} assure that 
\begin{align}
\label{eq:X:Phi}
\begin{split}
&\P\Big(X_T^{(1)}=  \Phi \big((W_s)_{s \in [0,T]}\big) \Big)= \P\Big(Y_T^{(1)} + \Xi=  \Phi \big((W_s)_{s \in [0,T]}\big) \Big)\\
&= \P\Big(Y_T^{(1)} =  \Phi \big((W_s)_{s \in [0,T]}\big) - \Xi \Big)= \P\Big(Y_T^{(1)} =  \phi \big((W_s)_{s \in [0,T]}\big)  \Big) =1.
\end{split}
\end{align}
Moreover, note that \eqref{eq:Y:xi} and \eqref{eq:inf:Y}   show that
\begin{align}
\label{eq:X:inf}
\begin{split}
&\inf_{\substack{u \colon C([0,a] \cup [b,T], \R)  \to \R \\  \text{ measurable}}}\E \! \left[\big| X_T^{(1)} - u((W_s)_{s \in [0,a] \cup [b,T]})\big|\right] \\
&= \inf_{\substack{u \colon C([0,a] \cup [b,T], \R)  \to \R \\  \text{ measurable}}}\E \! \left[\big| Y_T^{(1)}+ \Xi - u((W_s)_{s \in [0,a] \cup [b,T]})\big|\right]\\
& = \inf_{\substack{u \colon C([0,a] \cup [b,T], \R)  \to \R \\  \text{ measurable}}}\E \! \left[\big| Y_T^{(1)} - u((W_s)_{s \in [0,a] \cup [b,T]})\big|\right] \geq \delta_n.
\end{split}
\end{align}
Next observe that the fact that $f, g, \psi \in C^{\infty}(\R,\R)$, the fact that $\sup_{x \in \R} (|f(x)|+ |g(x)|) < \infty$, and \eqref{eq:M}--\eqref{eq:mu} ensure that $\mu, \sigma  \in C^{\infty}(\R^d, \R^d)$ and 
\begin{align}
\sup_{x \in \R^d} (\|\mu(x)\|_{\R^d} + \|\sigma(x)\|_{\R^d}) < \infty.
\end{align}
Combining this with \eqref{eq:X:Phi} and \eqref{eq:X:inf} completes the proof of Lemma~\ref{lem:main}.
\end{proof}

The next result, Theorem~\ref{thm:main} below,  extends the result of  Lemma~\ref{lem:main} by allowing the driving Brownian motion to be multidimensional.

\begin{theorem}
\label{thm:main}
Let $T \in (0, \infty)$, $\tau \in (0,T)$, $d \in \{2,3, \ldots\}$, $\xi \in \R^d$, $m\in \N$, $(\varepsilon_n)_{n \in \N} \subseteq (0, \tau]$, $(\delta_n)_{n \in \N} \subseteq \R$ satisfy $\limsup_{n \to \infty} \delta_n \leq 0$. Then there exist infinitely often differentiable and globally bounded functions $\mu \colon \R^d \to \R^d$ and $\sigma \colon \R^{d} \to \R^{d \times m}$ such that for every probability space $(\Omega, \mathcal{F}, \P)$, every normal filtration $\mathbb{F}=(\mathbb{F}_t)_{t \in [0,T]}$ on $(\Omega, \mathcal{F}, \P)$, every standard $(\Omega, \mathcal{F}, \P, \mathbb{F})$-Brownian motion $W \colon  [0,T] \times \Omega \to \R^m$, every continuous $\mathbb{F}$-adapted stochastic process $X=(X^{(1)}, \ldots, X^{(d)}) \colon [0,T] \times \Omega \to \R^d$ with $\forall \, t \in [0,T] \colon \P(X_t = \xi + \int_0^t \mu(X_s) \, ds + \int_0^t \sigma(X_s) \, dW_s)=1$, and every $n \in \N$ it holds that 
\begin{align}
\label{eq:prop:main}
\begin{split}
\inf_{\substack{a,b \in [0, \tau], \\
b-a \geq \varepsilon_n }}   \inf_{\substack{u \colon C([0,a] \cup [b,T], \R^m)  \to \R \\  \text{ measurable}}}\E \! \left[\big| X_T^{(1)} - u((W_s)_{s \in [0,a] \cup [b,T]})\big|\right] \geq \delta_n.
\end{split}
\end{align}
\end{theorem}
\begin{proof}[Proof of Theorem~\ref{thm:main}]
Throughout this proof assume w.l.o.g.~that 	$m \geq 2$ (otherwise \eqref{eq:prop:main} follows from Lemma \ref{lem:main}), let $\Phi \colon C([0,T],\R) \to \R$ and $\mu, \Sigma \in C^{\infty}(\R^d,\R^d)$   be measurable functions which satisfy that for every probability space $(\Omega, \mathcal{F}, \P)$, every normal filtration $\mathbb{F}=(\mathbb{F}_t)_{t \in [0,T]}$ on $(\Omega, \mathcal{F}, \P)$, every standard $(\Omega, \mathcal{F}, \P, \mathbb{F})$-Brownian motion $W \colon  [0,T] \times \Omega \to \R$, every continuous $\mathbb{F}$-adapted stochastic process $X=(X^{(1)}, \ldots, X^{(d)}) \colon [0,T] \times \Omega \to \R^d$ with $\forall \, t \in [0,T] \colon \P(X_t = \xi + \int_0^t \mu(X_s) \, ds + \int_0^t \Sigma(X_s) \, dW_s)=1$, and every $n \in \N$ it holds that 
\begin{align}
\label{eq:X1:meas}
\P\Big(X_T^{(1)}=  \Phi \big((W_s)_{s \in [0,T]}\big) \Big)=1,
\end{align}  
\begin{align}
\label{eq:m:1}
\begin{split}
\inf_{\substack{a,b \in [0, \tau], \\
b-a \geq \varepsilon_n }}   \inf_{\substack{u \colon C([0,a] \cup [b,T], \R)  \to \R \\  \text{ measurable}}}\E \! \left[\big| X_T^{(1)} - u((W_s)_{s \in [0,a] \cup [b,T]})\big|\right] \geq \delta_n,
\end{split}
\end{align}
and $\sup_{x \in \R^d} (\|\mu(x)\|_{\R^d}+ \|\Sigma(x)\|_{\R^d}) < \infty$ 
(Lemma~\ref{lem:main} assures that such functions do indeed exist), let $\sigma \colon \R^d \to \R^{d \times m}$ be the function which satisfies for all $x \in \R^d$, $y= (y_1, \ldots, y_m) \in \R^m$ that 
\begin{align}
\label{eq:sigma}
\sigma(x) y= y_1 \Sigma(x),
\end{align}
let $(\Omega, \mathcal{F}, \P)$ be a probability space, let  $\mathbb{F}=(\mathbb{F}_t)_{t \in [0,T]}$ be a normal filtration on $(\Omega, \mathcal{F}, \P)$,   let $W = (W^{(1)}, \ldots, W^{(m)}) \colon$  $[0,T] \times \Omega \to \R^m$ be a standard $(\Omega, \mathcal{F}, \P, \mathbb{F})$-Brownian motion, 
let $X=(X^{(1)}, \ldots, X^{(d)}) \colon [0,T] \times \Omega \to \R^d$ be a  continuous $\mathbb{F}$-adapted stochastic process which satisfies for all $t \in [0,T]$ that 
\begin{align}
\label{eq:X:Sigma}
\P\Big(X_t = \xi + \smallint\nolimits_0^t \mu(X_s) \, ds + \smallint\nolimits_0^t \sigma(X_s) \, dW_s\Big)=1,
\end{align}
let  $n \in \N$, $a,b \in [0,\tau]$ be real numbers with $b-a \geq \varepsilon_n$, let $u \colon C([0,a] \cup [b,T], \R^m)  \to \R$ be a measurable function, let $\tilde{W} = (\tilde{W}^{(1)}, \ldots, \tilde{W}^{(m)})\colon  \Omega \to C([0,T],\R^m)$ be the function which satisfies for all $\omega \in \Omega$, $t \in [0,T]$ that
\begin{align}
\label{eq:tilde:W}
(\tilde{W}(\omega))(t) = W_t(\omega),
\end{align}	
let $\Psi \colon C([0,T], \R ) \to C([0,a]\cup [b,T], \R)$  be the  function which satisfies for all $f \in C([0,T],\R)$ that $\Psi(f) = f|_{[0,a]\cup[b,T]}$, and for every $(v_1, \ldots, v_{m-1}) \in C([0,a] \cup[b,T], \R^{m-1})$ let $\tilde{u}_{v_1,\ldots,v_{m-1}} \colon C([0,a]\cup[b,T], \R) \to \R$ be the function which satisfies for all $v \in C([0,a]\cup[b,T], \R)$ that 
\begin{align}
\tilde{u}_{v_1,\ldots,v_{m-1}} (v)= u (v,v_1,\ldots,v_{m-1}).
\end{align}	
Observe that \eqref{eq:X:Sigma} and \eqref{eq:sigma} demonstrate that for all $t \in [0,T]$ it holds that
\begin{align}
\label{eq:X:sigma}
\P\Big(X_t = \xi + \textstyle \int_0^t \mu(X_s) \, ds + \int_0^t \Sigma(X_s) \, dW_s^{(1)}\Big)=1.
\end{align}
This,  \eqref{eq:X1:meas}, and \eqref{eq:tilde:W} assure that
\begin{align}
\label{eq:XT:meas}
\P\Big(X_T^{(1)}=  \Phi \big((W_s^{(1)})_{s \in [0,T]}\big)  = \Phi \big(\tilde{W}^{(1)}\big)\Big)=1.
\end{align} 
Next note that the fact that $\Sigma \in C^{\infty}(\R^d, \R^d)$, the fact that $\sup_{x \in \R^d} \|\Sigma(x)\|_{\R^{d}} < \infty$,  and \eqref{eq:sigma}  yield that  $\sigma \in C^{\infty} (\R^d, \R^{d \times m})$  and 
\begin{align}
\label{eq:sigma:norm}
\sup_{x \in \R^d} \|\sigma(x)\|_{\R^{d\times m}} = \sup_{x \in \R^d} \|\Sigma(x)\|_{\R^{d}} < \infty.
\end{align}
In addition, observe that 
\begin{align}
\begin{split}
u((W_s)_{s \in [0,a] \cup [b,T]}) &= u\big(\Psi(\tilde{W}^{(1)}), \ldots, \Psi(\tilde{W}^{(m)})\big)\\ 
&= \tilde{u}_{\Psi(\tilde{W}^{(2)}), \ldots, \Psi(\tilde{W}^{(m)} )} \big(\Psi(\tilde{W}^{(1)})\big).
\end{split}
\end{align}
Combining this with \eqref{eq:XT:meas} shows that 
\begin{align}
\begin{split}
&\E \! \left[\big| X_T^{(1)} - u((W_s)_{s \in [0,a] \cup [b,T]})\big|\right] = \E \! \left[\big| \Phi(\tilde{W}^{(1)}) - u((W_s)_{s \in [0,a] \cup [b,T]})\big|\right]\\
&  = \E \! \left[\big| \Phi(\tilde{W}^{(1)}) -\tilde{u}_{\Psi(\tilde{W}^{(2)}), \ldots, \Psi(\tilde{W}^{(m)} )} \big(\Psi(\tilde{W}^{(1)})\big)\big|\right] \\
& =\int_{\Omega} \big| \Phi(\tilde{W}^{(1)}(\omega)) -\tilde{u}_{\Psi(\tilde{W}^{(2)}(\omega)), \ldots, \Psi(\tilde{W}^{(m)}(\omega) )} \big(\Psi(\tilde{W}^{(1)}(\omega))\big)\big| \, \P(d\omega)\\
& = \int_{C([0,T],\R)} \ldots \int_{C([0,T],\R)} \big| \Phi(w_1) -\tilde{u}_{\Psi(w_2), \ldots, \Psi(w_m )} \big(\Psi(w_1)\big)\big| \\
&\quad \quad  \tilde{W}^{(1)}(\P)_{\mathcal{B}(C([0,T],\R))}(dw_1) \ldots \tilde{W}^{(m)}(\P)_{\mathcal{B}(C([0,T],\R))}(dw_m) \\
& = \int_{C([0,T],\R)} \ldots \int_{C([0,T],\R)}  \E \! \left[\big| \Phi(\tilde{W}^{(1)}) -\tilde{u}_{\Psi(w_2), \ldots, \Psi(w_m)} \big(\Psi(\tilde{W}^{(1)})\big)\big|\right] \\
&\quad \quad  \tilde{W}^{(2)}(\P)_{\mathcal{B}(C([0,T],\R))}(dw_2) \ldots \tilde{W}^{(m)}(\P)_{\mathcal{B}(C([0,T],\R))}(dw_m).
\end{split}
\end{align}
This, \eqref{eq:XT:meas}, \eqref{eq:X:sigma},  and \eqref{eq:m:1} ensure that 
\begin{align}
&\E \! \left[\big| X_T^{(1)} - u((W_s)_{s \in [0,a] \cup [b,T]})\big|\right]  \nonumber \\ 
& \geq \int_{C([0,T],\R)} \ldots \int_{C([0,T],\R)} \left[ \inf_{\substack{v \colon C([0,a] \cup [b,T], \R)  \to \R \\  \text{ measurable}}}  \E \! \left[\big| \Phi(\tilde{W}^{(1)}) -v \big(\Psi(\tilde{W}^{(1)})\big)\big|\right] \right] \nonumber\\
&\quad \quad  \tilde{W}^{(2)}(\P)_{\mathcal{B}(C([0,T],\R))}(dw_2) \ldots \tilde{W}^{(m)}(\P)_{\mathcal{B}(C([0,T],\R))}(dw_m) \nonumber \\
& = \inf_{\substack{v \colon C([0,a] \cup [b,T], \R)  \to \R \\  \text{ measurable}}}  \E \! \left[\big| \Phi(\tilde{W}^{(1)}) -v \big(\Psi(\tilde{W}^{(1)})\big)\big|\right] \nonumber  \\
& = \inf_{\substack{v \colon C([0,a] \cup [b,T], \R)  \to \R \\  \text{ measurable}}} \E \! \left[\big| X_T^{(1)} - v((W_s^{(1)})_{s \in [0,a] \cup [b,T]})\big|\right] \geq \delta_n.
\end{align}
Combining this with  \eqref{eq:sigma:norm} completes the proof of Theorem~\ref{thm:main}.  
\end{proof}
Next we strengthen the result of Theorem~\ref{thm:main} to strong approximations which may additionally use finitely many evaluations of the Brownian path.

\begin{cor}
\label{cor:W}
Let $T \in (0, \infty)$, $\tau \in (0,T)$, $d \in \{2,3, \ldots\}$, $\xi \in \R^d$, $m \in \N$,  $(\varepsilon_n)_{n \in \N} \subseteq (0, \tau]$, $(\delta_n)_{n \in \N} \subseteq \R$ satisfy $\limsup_{n \to \infty} \delta_n \leq 0$. Then there exist infinitely often differentiable and globally bounded functions $\mu \colon \R^d \to \R^d$ and $\sigma \colon \R^{d} \to \R^{d \times m}$ such that for every probability space $(\Omega, \mathcal{F}, \P)$, every normal filtration $\mathbb{F}=(\mathbb{F}_t)_{t \in [0,T]}$ on $(\Omega, \mathcal{F}, \P)$, every standard $(\Omega, \mathcal{F}, \P, \mathbb{F})$-Brownian motion $W \colon  [0,T] \times \Omega \to \R^m$, every continuous $\mathbb{F}$-adapted stochastic process $X=(X^{(1)}, \ldots, X^{(d)}) \colon [0,T] \times \Omega \to \R^d$ with $\forall \, t \in [0,T] \colon \P(X_t = \xi + \int_0^t \mu(X_s) \, ds + \int_0^t \sigma(X_s) \, dW_s)=1$, and every $n \in \N$ it holds that 
\begin{align}
\inf_{\substack{a,b \in [0, \tau], \\
b-a \geq \varepsilon_n }} \inf_{t_1,\ldots,t_n \in [0,T] }    \inf_{\substack{u \colon C([0,a] \cup [b,T], \R^m) \times (\R^m)^n  \to \R \\  \text{ measurable}}}\E  \Big[\big| X_T^{(1)} - u((W_s)_{s \in [0,a] \cup [b,T]},W_{t_1},\ldots,W_{t_n})\big|\Big] \geq \delta_n.
\end{align}
\end{cor}
\begin{proof}[Proof of Corollary~\ref{cor:W}]
Note that Theorem~\ref{thm:main} (with $T=T$, $\tau=\tau$, $d=d$, $\xi=\xi$, $m =m$, $\varepsilon_n = \tfrac{\varepsilon_n}{(n+1)}$, $\delta_n =\delta_n$ for $n \in \N$ in the notation of Theorem~\ref{thm:main})  proves that there exist infinitely often differentiable and globally bounded functions $\mu \colon \R^d \to \R^d$ and $\sigma \colon \R^{d} \to \R^{d \times m}$ such that for every probability space $(\Omega, \mathcal{F}, \P)$, every normal filtration $\mathbb{F}=(\mathbb{F}_t)_{t \in [0,T]}$ on $(\Omega, \mathcal{F}, \P)$, every standard $(\Omega, \mathcal{F}, \P, \mathbb{F})$-Brownian motion $W \colon  [0,T] \times \Omega \to \R^m$, every continuous $\mathbb{F}$-adapted stochastic process $X=(X^{(1)}, \ldots, X^{(d)}) \colon [0,T] \times \Omega \to \R^d$ with $\forall \, t \in [0,T] \colon \P(X_t = \xi + \int_0^t \mu(X_s) \, ds + \int_0^t \sigma(X_s) \, dW_s)=1$, and every $n \in \N$ it holds that 
\begin{align}
\begin{split}
&\inf_{\substack{a,b \in [0, \tau], \\
b-a \geq \varepsilon_n }} \inf_{t_1,\ldots,t_n \in [0,T] }    \inf_{\substack{u \colon C([0,a] \cup [b,T], \R^m) \times (\R^m)^n  \to \R \\  \text{ measurable}}}\E  \Big[\big| X_T^{(1)} - u((W_s)_{s \in [0,a] \cup [b,T]},W_{t_1},\ldots,W_{t_n})\big|\Big] \\
&\geq \inf_{\substack{a,b \in [0, \tau], \\
b-a \geq \nicefrac{\varepsilon_n}{(n+1)} }}    \inf_{\substack{u \colon C([0,a] \cup [b,T], \R^m)  \to \R \\  \text{ measurable}}}\E  \Big[\big| X_T^{(1)} 
- u((W_s)_{s \in [0,a] \cup [b,T]})\big|\Big] \geq  \delta_n.
\end{split}
\end{align}
The proof of Corollary~\ref{cor:W} is thus completed.
\end{proof}

%
%

\bibliographystyle{acm}
\bibliography{bibfile}

\def\polhk#1{\setbox0=\hbox{#1}{\ooalign{\hidewidth
  \lower1.5ex\hbox{`}\hidewidth\crcr\unhbox0}}}
\begin{thebibliography}{10}

\bibitem{Gyoengy1998}
{\sc Gy{\"o}ngy, I.}
\newblock A note on {E}uler's approximations.
\newblock {\em Potential Anal. 8}, 3 (1998), 205--216.

\bibitem{Hairer2015}
{\sc Hairer, M., Hutzenthaler, M., and Jentzen, A.}
\newblock Loss of regularity for {K}olmogorov equations.
\newblock {\em Ann. Probab. 43}, 2 (2015), 468--527.

\bibitem{Higham2002}
{\sc Higham, D.~J., Mao, X., and Stuart, A.~M.}
\newblock Strong convergence of {E}uler-type methods for nonlinear stochastic
  differential equations.
\newblock {\em SIAM J. Numer. Anal. 40}, 3 (2002), 1041--1063 (electronic).

\bibitem{Hu1996}
{\sc Hu, Y.}
\newblock Semi-implicit {E}uler-{M}aruyama scheme for stiff stochastic
  equations.
\newblock In {\em Stochastic analysis and related topics, {V} ({S}ilivri,
  1994)}, vol.~38. Birkh\"auser Boston, Boston, MA, 1996, pp.~183--202.

\bibitem{Hutzenthaler2015}
{\sc Hutzenthaler, M., and Jentzen, A.}
\newblock Numerical approximations of stochastic differential equations with
  non-globally {L}ipschitz continuous coefficients.
\newblock {\em Mem. Amer. Math. Soc. 4\/} (2015), 1--112.

\bibitem{hjk11}
{\sc Hutzenthaler, M., Jentzen, A., and Kloeden, P.~E.}
\newblock Strong and weak divergence in finite time of {E}uler's method for
  stochastic differential equations with non-globally {L}ipschitz continuous
  coefficients.
\newblock {\em Proc. R. Soc. Lond. Ser. A Math. Phys. Eng. Sci. 467\/} (2011),
  1563--1576.

\bibitem{HutzenthalerJentzenKloeden2012}
{\sc Hutzenthaler, M., Jentzen, A., and Kloeden, P.~E.}
\newblock Strong convergence of an explicit numerical method for {SDE}s with
  non-globally {L}ipschitz continuous coefficients.
\newblock {\em Ann. Appl. Probab. 22}, 4 (2012), 1611--1641.

\bibitem{HutzenthalerJentzenKloeden2013}
{\sc Hutzenthaler, M., Jentzen, A., and Kloeden, P.~E.}
\newblock Divergence of the multilevel {M}onte {C}arlo {E}uler method for
  nonlinear stochastic differential equations.
\newblock {\em Ann. Appl. Probab. 23}, 5 (2013), 1913--1966.

\bibitem{jentzen2016slow}
{\sc Jentzen, A., M{\"u}ller-Gronbach, T., and Yaroslavtseva, L.}
\newblock On stochastic differential equations with arbitrary slow convergence
  rates for strong approximation.
\newblock {\em Commun. Math. Sci. 14\/} (2016), 1477--1500.

\bibitem{Mueller-Gronbach2008}
{\sc M{\"u}ller-Gronbach, T., and Ritter, K.}
\newblock Minimal errors for strong and weak approximation of stochastic
  differential equations.
\newblock In {\em Monte {C}arlo and quasi-{M}onte {C}arlo methods 2006}.
  Springer, Berlin, 2008, pp.~53--82.

\bibitem{MullerGronbachYaroslavtseva2017}
{\sc M{\"u}ller-Gronbach, T., and Yaroslavtseva, L.}
\newblock On sub-polynomial lower error bounds for quadrature of {SDE}s with
  bounded smooth coefficients.
\newblock {\em Stochastic Analysis and Applications 0}, 0 (2017), 1--29.

\bibitem{Sabanis2013}
{\sc Sabanis, S.}
\newblock {A note on tamed Euler approximations}.
\newblock {\em Electron. Commun. Probab. 18\/} (2013), no. 47, 1--10.

\bibitem{sabanis2016}
{\sc Sabanis, S.}
\newblock Euler approximations with varying coefficients: The case of
  superlinearly growing diffusion coefficients.
\newblock {\em Ann. Appl. Probab. 26}, 4 (2016), 2083--2105.

\bibitem{taononlinear}
{\sc Tao, T.}
\newblock {\em Nonlinear Dispersive Equations: Local and Global Analysis},
  vol.~106 of {\em CBMS Regional Conference Series in Mathematics}.
\newblock Published for the Conference Board of the Mathematical Sciences by
  the American Mathematical Society, Providence, RI, 2006.

\bibitem{TretyakovZhang2013}
{\sc Tretyakov, M.~V., and Zhang, Z.}
\newblock {A fundamental mean-square convergence theorem for SDEs with locally
  Lipschitz coefficients and its applications}.
\newblock {\em SIAM J. Numer. Anal. 51}, 6 (2013), 3135--3162.

\bibitem{WangGan2013}
{\sc Wang, X., and Gan, S.}
\newblock {The tamed Milstein method for commutative stochastic differential
  equations with non-globally Lipschitz continuous coefficients}.
\newblock {\em J. Difference Equ. Appl. 19}, 3 (2013), 466--490.

\bibitem{Yaroslavtseva2016}
{\sc {Yaroslavtseva}, L.}
\newblock On non-polynomial lower error bounds for adaptive strong
  approximation of {SDE}s.
\newblock {\em arXiv:1609.08073\/} (2016), 19 pages.

\end{thebibliography}

\end{document}